\documentclass{amsart}
\usepackage{amsmath}
\usepackage{paralist}
\usepackage{graphics} 
\usepackage{epsfig}
\usepackage{graphicx}  
\usepackage{epstopdf}
\usepackage{amssymb,amsthm}
\usepackage{graphicx}
\usepackage{color}
\usepackage[usenames,dvipsnames,svgnames,table]{xcolor}
\usepackage{tikz}
\usepackage{comment}
\usepackage{float}
\usepackage[utf8]{inputenc}
\usepackage[colorlinks=true]{hyperref}
\hypersetup{urlcolor=blue, citecolor=red}

\newtheorem{theorem}{Theorem}[section]

\newtheorem{lemma}[theorem]{Lemma}
\newtheorem{proposition}{Proposition}

\theoremstyle{definition}
\newtheorem{definition}[theorem]{Definition}
\newtheorem{remark}{Remark}

\newtheorem*{question}{Question}

\renewcommand{\epsilon}{\varepsilon}
\renewcommand{\phi}{\varphi}
\renewcommand{\theta}{\vartheta}

\title[Parabolic Arcs of the Multicorns]{Parabolic arcs of the multicorns: real-analyticity of Hausdorff dimension, and singularities of $\mathrm{Per}_n(1)$ Curves}

\author{Sabyasachi Mukherjee}



\email{sabya@math.stonybrook.edu, s.mukherjee@jacobs-university.de}
 
\thanks{The author was supported by Deutsche Forschungsgemeinschaft DFG}

\begin{document}
\maketitle

%\centerline{\scshape Sabyasachi Mukherjee}
\medskip
{\footnotesize

 \centerline{Jacobs University Bremen}
   \centerline{Campus Ring 1}
   \centerline{Bremen 28759, Germany}
}

\medskip
{\footnotesize

 \centerline{Institute for Mathematical Sciences}
   \centerline{Stony Brook University} 
   \centerline{Stony Brook, 11794, NY, USA}
}

\date{\today}

\begin{abstract}
The boundaries of the hyperbolic components of odd period of the multicorns contain real-analytic arcs consisting of quasi-conformally conjugate parabolic parameters. One of the main results of this paper asserts that the Hausdorff dimension of the Julia sets is a real-analytic function of the parameter along these parabolic arcs. This is achieved by constructing a complex one-dimensional quasiconformal deformation space of the parabolic arcs which are contained in the dynamically defined algebraic curves $\mathrm{Per}_n(1)$ of a suitably complexified family of polynomials. As another application of this deformation step, we show that the dynamically natural parametrization of the parabolic arcs has a non-vanishing derivative at all but (possibly) finitely many points. 

We also look at the algebraic sets $\mathrm{Per}_n(1)$ in various families of polynomials, the nature of their singularities, and the `dynamical' behavior of these singular parameters.
\end{abstract}

\section{Introduction}\label{intro}

The multicorns are the connectedness loci of unicritical antiholomorphic polynomials. Any unicritical antiholomorphic polynomial, up to an affine change of coordinates, can be written in the form $f_c(z) = \overline{z}^d + c$, for some $c \in \mathbb{C}$ and $d \geq 2$. In analogy to the holomorphic case, the set of all points which remain bounded under all iterations of $f_c$ is called the \emph{filled-in Julia set} $K(f_c)$. The boundary of the filled-in Julia set is defined to be the \emph{Julia set} $J(f_c)$ and the complement of the Julia set is defined to be its \emph{Fatou set} $F(f_c)$. This leads, as in the holomorphic case, to the notion of \emph{connectedness locus} of degree $d$ unicritical antiholomorphic polynomials:

\begin{definition}[Multicorns]
The \emph{multicorn} of degree $d$ is defined as $\mathcal{M}^{\ast}_d := \lbrace c \in \mathbb{C} : K(f_c)$ is connected$\rbrace$. The multicorn of degree $2$ is called the tricorn.
\end{definition}

It follows from classical works of Bowen and Ruelle \cite{Ru,Zi} that the Hausdorff dimension of the Julia set depends real-analytically on the parameter within every hyperbolic component of $\mathcal{M}_d^*$. Ruelle's proof makes essential use of the fact that hyperbolic rational maps are expanding (this allows one to use the powerful machinery of thermodynamic formalism), and the Julia sets of hyperbolic rational maps move holomorphically inside every hyperbolic component.

The boundary of every hyperbolic component of \emph{odd} period of $\mathcal{M}_d^*$ is a simple closed curve consisting of exactly $d+1$ double parabolic parameters (also called `cusp points') as well as $d+1$ parabolic arcs, each connecting two double parabolics. Moreover, any two parameters on a given parabolic arc have quasiconformally conjugate dynamics \cite{MNS}. Since parabolic maps have a certain weak expansion property, and since there are real-analytic arcs of quasiconformally conjugate parabolic parameters on the boundary of every hyperbolic component of odd period of $\mathcal{M}_d^*$, it is natural to ask whether the Hausdorff dimension of the Julia set depends real-analytically on the parameter along these parabolic arcs. 

However, an apparent obstruction to proving real-analyticity of Hausdorff dimension of the Julia set along the parabolic arcs is that the parabolic arcs are real one-dimensional curves, and hence one cannot find a holomorphic motion of the Julia sets (of the parabolic parameters) within the family of unicritical antiholomorphic polynomials. We circumvent this problem by constructing a strictly larger quasiconformal deformation class of the odd period simple parabolic parameters (also called `non-cusp' parabolic parameters) of the multicorns so that the deformation is no longer contained in the family of unicritical antiholomorphic polynomials, but lives in a bigger family of holomorphic polynomials. This helps us to embed the parabolic arcs (which are real one-dimensional curves) in a complex one-dimensional family of quasiconformally conjugate parabolic maps. This proves the existence of holomorphic motion of the Julia sets under consideration (better yet, this proves structural stability of the odd period non-cusp parabolic maps along a suitable algebraic curve). This is performed in Section \ref{non-singular} by varying the critical Ecalle height over a bi-infinite strip by a quasiconformal deformation argument.

Having the complexification of the parabolic arcs (i.e. complex analytic parameter dependence of the persistently parabolic maps) at our disposal, we can apply results on real-analyticity of Hausdorff dimension of Julia sets of analytic families of meromorphic functions, as developed in \cite{SU}, to our setting. The following theorem, which is proved in Section \ref{real-anal}, can be naturally thought of as a version of Ruelle's theorem on the boundaries of hyperbolic components:

\begin{theorem}[Real-analyticity of HD Along Parabolic Arcs]\label{real-anal HD}
Let $\mathcal{C}$ be a parabolic arc of $\mathcal{M}_d^*$ and let $c: \mathbb{R} \rightarrow \mathcal{C}, h \mapsto c(h)$ be its critical Ecalle height parametrization. Then the function
\begin{align*}
\mathbb{R} \ni h \mapsto \mathrm{HD}(J(f_{c(h)}))
\end{align*}
is real-analytic.
\end{theorem}

\begin{figure}[ht!]
\begin{center}
\includegraphics[scale=0.2]{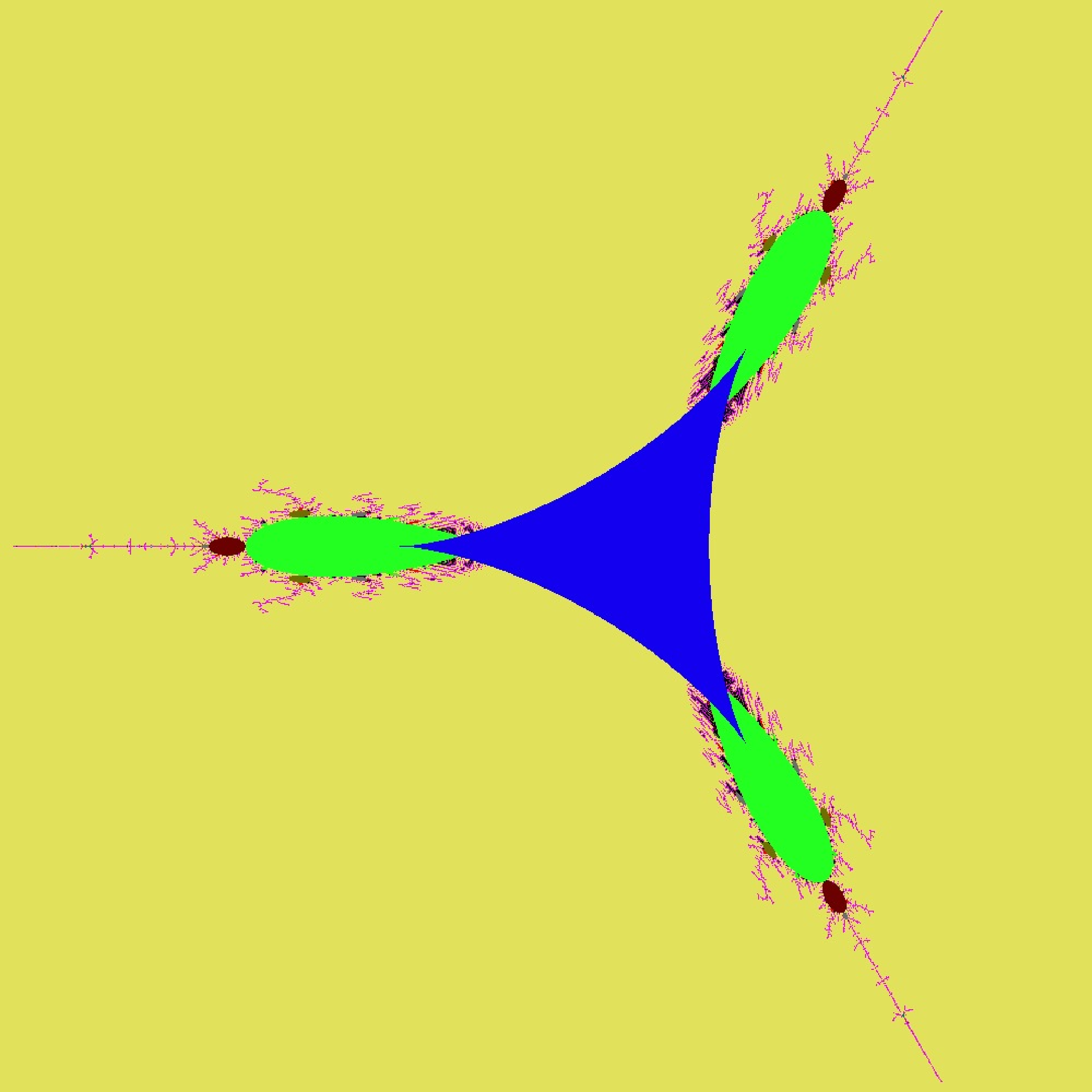}
\caption{$\mathcal{M}_2^*$, also known as the tricorn and the parabolic arcs on the boundary of the hyperbolic component of period $1$ (in blue)}
\label{tricorn}
\end{center}
\end{figure}

It will transpire from the course of the proof that in good situations, the real-analyticity of Hausdorff dimension holds more generally on certain regions of the parabolic curves $\mathrm{Per}_n(1)$ (see \cite{M3} for the definition of the $\mathrm{Per}$ curves) of various other families of polynomials.

As a by-product of the quasiconformal deformation step, we prove that the Ecalle height parametrization of the parabolic arcs of the multicorns is non-singular at all but possibly finitely many points.

\begin{theorem}\label{almost non-singular}
Let $\mathcal{C}$ be a parabolic arc of odd period of $\mathcal{M}_d^*$ and $c : \mathbb{R} \rightarrow \mathcal{C}$ be its critical Ecalle height parametrization. Then, there exists a holomorphic map $\phi : S=\lbrace w = u + iv \in \mathbb{C} : \vert v \vert < \frac{1}{4} \rbrace \rightarrow \mathbb{C}$ such that:
\begin{enumerate}
\item The map $\phi$ agrees with the map $c$ on $\mathbb{R}$.

\item  For all but possibly finitely many $x \in \mathbb{R}$, $c'(x) = \phi'(x) \neq 0$.
\end{enumerate} 
In particular, the critical Ecalle height parametrization of any parabolic arc of $\mathcal{M}_d^*$ has a non-vanishing derivative at all but possibly finitely many points. 
\end{theorem}

In Section \ref{per_1_1}, we look at the curves $\mathrm{Per}_n(1)$ in biquadratic polynomials, and prove that each cusp point (odd period double parabolic parameter) is a singular point (with at least a double tangent) of $\mathrm{Per}_n(1)$. In Section \ref{sing_per_n_1}, we look at some more examples of the algebraic sets $\mathrm{Per}_n(1)$ (in various families of polynomials), and try to understand the nature of their singularities as well the `dynamical' behavior of these singular parameters. We do not prove any precise theorem here, we rather investigate some concrete examples, and give heuristic explanations of singularity/non-singularity, with a view towards a more general understanding of the topology of these algebraic sets.

This paper is a continuation of the author's investigation of the parameter spaces of unicritical antiholomorphic polynomials \cite{Sa,MNS,IM}. These parameter spaces have played an important role in the recent work by various people \cite{BBM1,BBM2,CFG,BeEr}. Therefore, notwithstanding the fact that the results of this paper are likely to hold in more general settings, it is worthwhile to record the ideas in the concrete case of unicritical antiholomorphic polynomials. It is worth mentioning that the present paper adds one more item to the list of topological differences between the multicorns and their holomorphic counterparts, the multibrot sets (these are the connectedness loci of unicritical holomorphic polynomials $z^d+c$). Clearly, the parameter dependence of the Hausdorff dimension of the Julia sets is far from regular on the boundary of the Mandelbrot set. It has been recently proved \cite{HS} that the multicorns are not locally connected, while the local connectivity of the Mandelbrot set is one of of the most prominent conjectures in one dimensional complex dynamics. In a recent work \cite{IM}, we proved that rational parameter rays at odd-periodic angles of the multicorns do not land, rather they accumulate on arcs of positive length in the parameter space. This is in stark contrast with the fact that every rational parameter ray of the multibrot sets lands at a unique parameter. In \cite{IM}, we also showed that the centers of the hyperbolic components, and the Misiurewicz parameters do not accumulate on the entire boundary of the multicorns. More such topological differences between the multibrot sets and the multicorns, including bifurcation along arcs, existence of real-analytic arcs of quasi-conformally equivalent parabolic parameters, discontinuity of landing points of dynamical rays, etc. can be found in \cite{MNS}.

\section{Antiholomorphic Fatou Coordinates, Equators and Ecalle heights}\label{Secbackground}

In this section, we recall some basic facts about the parameter spaces of unicritical antiholomorphic polynomials. One of the major differences between the multicorns and the multibrot sets (which are the connectedness loci of unicritical holomorphic polynomials $z^d+c$) is that the boundaries of odd period hyperbolic components of the multicorns consist only of parabolic parameters.

\begin{lemma}[Indifferent Dynamics of Odd Period]\label{LemOddIndiffDyn}  
The boundary of every hyperbolic component of odd period $k$ of $\mathcal{M}_d^*$ consists 
entirely of parameters having a parabolic cycle of exact period $k$. In appropriate local conformal coordinates, the $2k$-th iterate of such a map has the form 
$z\mapsto z+z^{q+1}+\ldots$ with $q\in\{1,2\}$. 
\end{lemma}

\begin{proof} 
See \cite[Lemma 2.8]{MNS}.
\end{proof}

This leads to the following classification of odd periodic parabolic points.

\begin{definition}[Parabolic Cusps]\label{DefCusp}
A parameter $c$ will be called a {\em cusp point} if it has a parabolic 
periodic point of odd period such that $q=2$ in the previous lemma. Otherwise, it is called a \emph{simple} (or non-cusp) parabolic parameter.
\end{definition}

In holomorphic dynamics, the local dynamics in attracting petals of parabolic periodic points is well-understood: there is a local coordinate $\zeta$ which conjugates the first-return dynamics to the form $\zeta\mapsto\zeta+1$ in a right half place \cite[Section~10]{M1new}. Such a coordinate $\zeta$ is called a \emph{Fatou coordinate}. Thus the quotient of the petal by the dynamics is isomorphic to a bi-infinite cylinder, called an \emph{Ecalle cylinder}. Note that Fatou coordinates are uniquely determined up to addition by a complex constant. 

In antiholomorphic dynamics, the situation is at the same time restricted and richer. Indifferent dynamics of odd period is always parabolic because for an indifferent periodic point of odd period $k$, the $2k$-th iterate is holomorphic with positive real multiplier, hence parabolic as described above. On the other hand, additional structure is given by the antiholomorphic intermediate iterate. 

\begin{lemma}
Suppose $z_0$ is a parabolic periodic point of odd period $k$ of $f_c$ with only one petal (i.e. $c$ is not a cusp) and $U$ is a periodic Fatou component with $z_0 \in \partial U$. Then there is an open subset $V \subset U$ with $z_0 \in \partial V$ and $f_c^{\circ k} (V) \subset V$ so that for every $z \in U$, there is an $n \in \mathbb{N}$ with $f_c^{\circ nk}(z)\in 
V$. Moreover, there is a univalent map $\Phi \colon V \to \mathbb{C}$ with $\Phi(f_c^{\circ k}(z)) = \overline{\Phi(z)}+1/2$, and $\Phi(V)$ contains a right half plane. This map $\Phi$ is unique up to horizontal translation. 
\end{lemma}
\begin{proof} 
See \cite[Lemma 2.3]{HS}.
\end{proof}

The map $\Phi$ will be called an  \emph{antiholomorphic Fatou coordinate} for the petal $V$. The antiholomorphic iterate interchanges both ends of the Ecalle cylinder, so it must fix one horizontal line around this cylinder (the \emph{equator}). The change of coordinate $\Phi$ has been so chosen that it maps the equator to the the real axis.  We will call the vertical Fatou coordinate the \emph{Ecalle height}. Its origin is the equator. The existence of this distinguished real line, or equivalently an intrinsic meaning to Ecalle height, is specific to antiholomorphic maps. 

The Ecalle height of the critical value of $f_c$ is called the \emph{critical Ecalle height} of $f_c$, and it plays a special role in antiholomorphic dynamics. The next theorem proves the existence of real-analytic arcs of non-cusp parabolic parameters on the boundaries of odd period hyperbolic components of the multicorns.

\begin{theorem}[Parabolic arcs]
Let $c_0$ be a parameter such that $f_{c_0}$ has a parabolic cycle of odd period, and suppose that $c_0$ is not a cusp. Then $c_0$ is on a parabolic arc in the  following sense: there  exists an injective real-analytic map $c:\mathbb{R}\rightarrow\mathcal{M}_d^*$ such that for all $h\in\mathbb{R}$, $c(h)$ is a non-cusp parabolic parameter, the critical Ecalle height of $f_{c(h)}$ is $h$, the antiholomorphic polynomials $f_{c(h)}$ have quasiconformally equivalent but conformally distinct dynamics, and $c_0$ is an interior point of the simple real-analytic arc $c(\mathbb{R})$.
\end{theorem}
\begin{proof} 
See \cite[Theorem 3.2]{MNS}.
\end{proof}

The simple real-analytic arc $c(\mathbb{R})$ constructed in the above theorem is called a \emph{parabolic arc}, and is denoted by $\mathcal{C}$. The parametrization of the parabolic arc $\mathcal{C}$ given in the previous theorem is called the \emph{critical Ecalle height parametrization} of $\mathcal{C}$. 

The structure of the hyperbolic components of odd period plays an important role in the global topology of the parameter spaces. 

\begin{theorem}[Boundaries Of Odd Period Hyperbolic Components]\label{Exactly d+1}
The boundary of every hyperbolic component of odd period of $\mathcal{M}_d^*$ is a simple
closed curve consisting of exactly $d + 1$ parabolic cusp points as well as $d+1$ parabolic arcs, each connecting two parabolic cusps.
\end{theorem}
\begin{proof}
See \cite[Theorem 1.2]{MNS}.
\end{proof}

We refer the readers to \cite{Na1,NS,HS,MNS,IM} for a more comprehensive account on the combinatorics and topology of the multicorns.

\section{Constructing an Analytic Family of Q.C. Deformations}\label{non-singular}

Throughout this section, we fix a parabolic arc $\mathcal{C}$ such that for each $c \in \mathcal{C}$, $f_c$ has a $k$-periodic parabolic cycle. Recall that there is a dynamically defined parametrization of $\mathcal{C}$, which is its critical Ecalle height parametrization $c : \mathbb{R} \rightarrow \mathcal{C}$. We will show that the polynomials on $\mathcal{C}$ can be quasi-conformally deformed to yield an analytic family of q.c. conjugate maps; in particular, they will be structurally stable on a suitable algebraic curve.

We embed our family $\{f_c(z) = \overline{z}^d + c, \hspace{1mm} c \in \mathbb{C}\}$ in the family of holomorphic polynomials $\mathcal{F}_d = \lbrace P_{a,b}(z) = \left(z^d + a \right)^d + b, \hspace{1mm} a, b \in \mathbb{C} \rbrace$. Since $f_c^{\circ 2} = P_{\overline{c},c}$, the connectedness locus $\mathcal{C}(\mathcal{F}_d)$ of this family intersects the slice $\lbrace a=\overline{b}\rbrace$ in $\mathcal{M}_d^*$.

It will be useful to have the following characterization of the elements of $\mathcal{F}_d$ amongst all monic centered polynomials of degree $d^2$. 

\begin{lemma}\label{characterization}
Let $P$ be a monic centered polynomial of degree $d^2$. Then the following are equivalent:
\begin{enumerate}
\item $P \in \mathcal{F}_d$ with $a \neq 0$.

\item $P$ has exactly $d+1$ distinct critical points $\lbrace \alpha_1, \alpha_2, \cdots, \alpha_{d+1} \rbrace$ with $deg_{\alpha_i}(P) = d$ for each $i$ and such that $P(\alpha_1) = P(\alpha_2) = \cdots = P(\alpha_d)$.
\end{enumerate}
\end{lemma}

\begin{proof}
$(1) \implies (2).$ The critical points of any $P \in \mathcal{F}_d$ are $0$ and the $d$ roots of the equation $z^d+a = 0$. Since $a \neq 0$, these points are all distinct and $P$ has local degree $d$ at each of them. The other property is immediate.

$(2) \implies (1).$ Let, $b = P(\alpha_1) = P(\alpha_2) = \cdots = P(\alpha_d)$. Since each $\alpha_i$ ($i = 1, 2, \cdots, d$) maps in a $d$-to-$1$ fashion to $b$ and the degree of $P$ is $d^2$, these must be all the pre-images of $b$. Therefore, $P(z) = p(z)^d +b$, where $p(z) = (z-\alpha_1)\cdots (z-\alpha_d)$ (here, we have used the fact that $P$ is monic). Note that the critical points of $P$ are precisely the zeroes and the critical points of $p$. Since we have used up the $d$ distinct zeroes of $p$ and are left with only one critical point of $P$, it follows that this critical point $\alpha_{d+1}$ must be the only critical point of $p$. Hence, $p$ must be unicritical. A brief computation (using the fact that $P$ is centered) now shows that $p(z) = z^d +a$, for some $a \in \mathbb{C}$. The fact that all the $\alpha_i$ are distinct tells that $a$ must be non-zero. Thus, we have shown that $P(z) = (z^d+a)^d+b$ with $a \in \mathbb{C}^*$, $b \in \mathbb{C}$. 
\end{proof}

\begin{remark}
Condition (2) of the previous lemma is preserved under topological conjugacies. 
\end{remark}

Before we can turn our attention to the algebraic curves that are the principal objects of study of this paper, we need to take a brief digression to algebra. In what follows, we will discuss the properties of discriminants and resultants of complex polynomials. 

For a monic complex polynomial $Q(z) = z^n+a_{n-1}z^{n-1}+\cdots+a_1z+a_0 \in \mathbb{C}[z]$ with roots $\{\beta_1, \cdots, \beta_n\}$, the discriminant of $Q(z)$ is defined as:

\begin{center}
$\displaystyle\textrm{disc}_z(Q(z)):=\prod_{i<j}(\beta_i-\beta_j)^2$
\end{center}

It is evident from the definition of the discriminant that $Q(z)$ has a multiple (or repeated) root if and only if $\textrm{disc}_z(Q(z))$ vanishes. Moreover, $\textrm{disc}_z(Q(z))$ is a symmetric function of the roots $\{\beta_1, \cdots, \beta_n\}$ of $Q(z)$, hence it can be expressed as a polynomial function of the \emph{elementary} symmetric functions in the variables $\{\beta_1, \cdots, \beta_n\}$. But the elementary symmetric functions in $\{\beta_1, \cdots, \beta_n\}$, up to sign, coincide with the coefficients $\{a_0, \cdots, a_{n-1}\}$ of $Q(z)$. Therefore, $\textrm{disc}_z(Q(z))$ can be written as a polynomial in the coefficients of $Q(z)$ (compare \cite[\S 14.6]{DuFo}).

In a similar vein, the \emph{resultant} of two univariate polynomials $P$ and $Q$ is a polynomial in the coefficients of $P$ and $Q$ which vanishes if and only if $P$ and $Q$ have a common root; i.e. if and only if $\mathrm{deg(gcd}(P,Q)) \geq 1$. It is sometimes desirable to generalize the notion of resultants to be able to predict the exact degree of the gcd of $P$ and $Q$ in terms of their coefficients. This can be done by the so-called \emph{subresultants}, which are polynomials in the coefficients of $P$ and $Q$, and whose order of vanishing tells us the degree of the gcd of the two polynomials $P$ and $Q$. We refer the readers to \cite[\S 4.2.2]{BPC} for the precise definition of subresultants.

The main property of subresultants that will be used in this paper is the following.
 
\begin{lemma}\label{subresultants}
For two univariate polynomials $P$ and $Q$ over a domain (of degree $p$ and $q$ respectively, such that $p > q$), $\mathrm{deg(gcd}(P,Q)) \geq j$ for some $0 \leq j \leq q$ if and only if $sRes_0(P,Q)=\cdots =sRes_{j-1}(P,Q)=0$, where each $sRes_i(P,Q)$ is a polynomial expression in the coefficients of $P$ and $Q$.
\end{lemma} 

\begin{proof}
See \cite[Proposition 4.25]{BPC}.
\end{proof}

We will now use the concept of discriminants to define parabolic curves in the parameter space of the family $\mathcal{F}_d$. Since every antiholomorphic polynomial $f_c$ on a parabolic arc $\mathcal{C}$ has a parabolic cycle of odd period $k$ and multiplier $1$, its holomorphic second iterate $f_c^{\circ 2} = P_{\overline{c},c}$ also has a $k$-periodic parabolic cycle of multiplier $1$ (note that as $k$ is odd, we have that $2$ and $k$ are relatively prime). Therefore, under the embedding $c \mapsto (\overline{c}, c)$ of the family $\{f_c\}_{c\in\mathbb{C}}$ of unicritical antiholomorphic polynomials into the family of holomorphic polynomials $\mathcal{F}_d$, the image of the parabolic arc $\mathcal{C}$ sits in the set of all $k$-periodic parabolic parameters of multiplier $1$. However, if $P_{a,b}$ has a $k$-periodic parabolic cycle of multiplier $1$, then $( P_{a,b}^{\circ k}(z) - z)$ has a multiple root. Therefore by definition of discriminants, the parabolic arc $\mathcal{C}$ embeds into the set:

\begin{center}
$P_k = \lbrace (a, b) \in \mathbb{C}^2 : \Psi(a,b) := \mathrm{disc}_z ( P_{a,b}^{\circ k}(z) - z) = 0 \rbrace$.
\end{center}

Note that $(P_{a,b}^{\circ k}(z) - z)\in \mathbb{C}[a,b][z]$; i.e. it is a polynomial in the variable $z$ with coefficients from the ring $\mathbb{C}[a,b]$. It now follows from our discussion on discriminants that $\mathrm{disc}_z (P_{a,b}^{\circ k}(z)-z)$ is a polynomial in the variables $a$, $b$ with complex coefficients. Hence $P_k$ is an algebraic curve in $\mathbb{C}^2$.

We are now in a position to prove the main technical lemma, which allows us to embed the parabolic arc $\mathcal{C}$ in a complex one-dimensional deformation space contained in $P_k$. The idea of the proof is as follows. For any $c_0 \in \mathcal{C}$, the second iterate $P_{\overline{c_0},c_0}=f_{c_0}^{\circ 2}$ of the antiholomorphic polynomial $f_{c_0}$ is a holomorphic polynomial with two infinite critical orbits. Both these critical orbits are attracted by the unique parabolic cycle of $P_{\overline{c_0},c_0}$, and hence we can choose two representatives of these two critical orbits in a fundamental domain of an attracting petal of $P_{\overline{c_0},c_0}$. The difference of the attracting Fatou coordinates of these two chosen representatives (the difference does not depend on the choice of the representatives as long as they are chosen from the same fundamental domain) turn out to be a conformal conjugacy invariant of the polynomial. This invariant will be called the \emph{Fatou vector} of the polynomial. We will deform the polynomial in such a way that the qualitative behavior of the dynamics remains the same, but the Fatou vector varies over a bi-infinite strip. Thus, each deformation would give us a topologically conjugate, but conformally different polynomial.

\begin{lemma}[Extending the Deformation]\label{qcdef}
Let $c_0 = c(0) \in \mathcal{C}$. There exists an injective holomorphic map
\begin{center}
$F : S=\lbrace w = u + iv \in \mathbb{C} : \vert v \vert < \frac{1}{4} \rbrace \rightarrow P_k$,\\
$w \mapsto (a(w) , b(w))$ 
\end{center}
with $F(0) = (\overline{c_0},c_0)$ such that for any $w \in \mathbb{R}$, $F(w) = (\overline{c(w)} , c(w))$ where $c : \mathbb{R} \rightarrow \mathcal{C}$ is the critical Ecalle height parametrization of the parabolic arc. Further, all the polynomials $P_{a(w),b(w)}$ have q.c.-conjugate (but not conformally conjugate) dynamics.
\end{lemma}

\begin{figure}[ht!]
\centering{\includegraphics[scale=0.5]{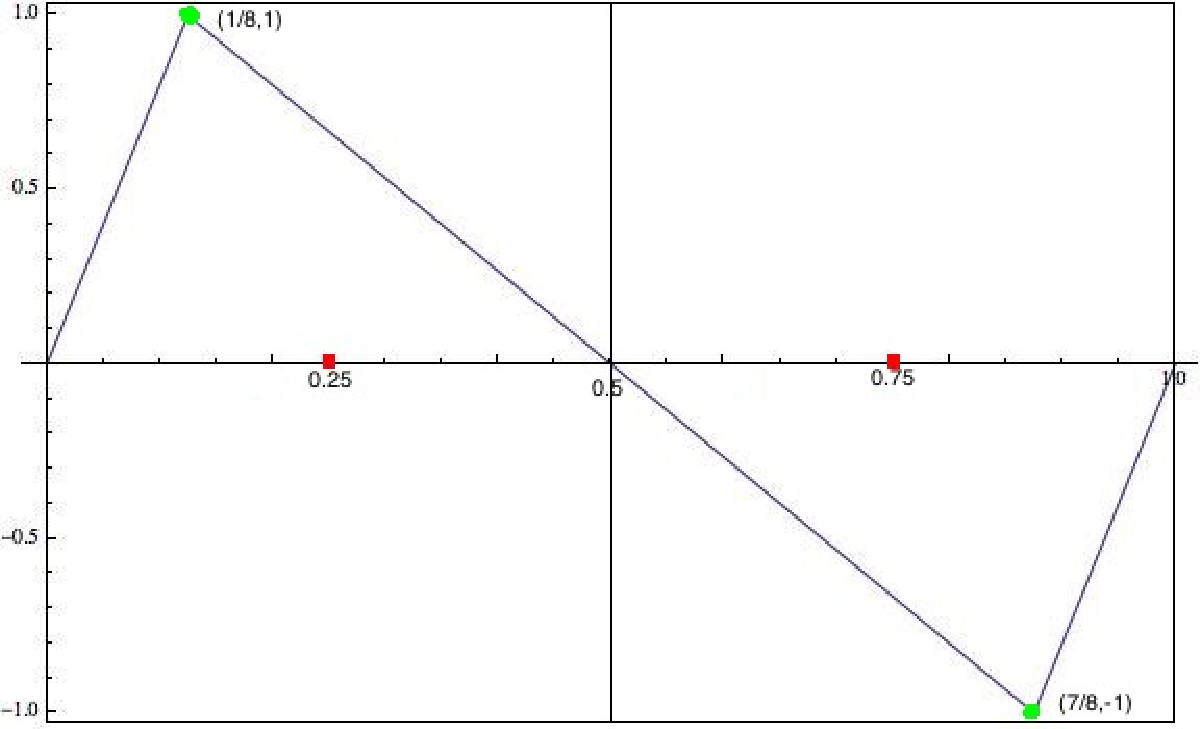}}

\centering{\includegraphics[scale=0.5]{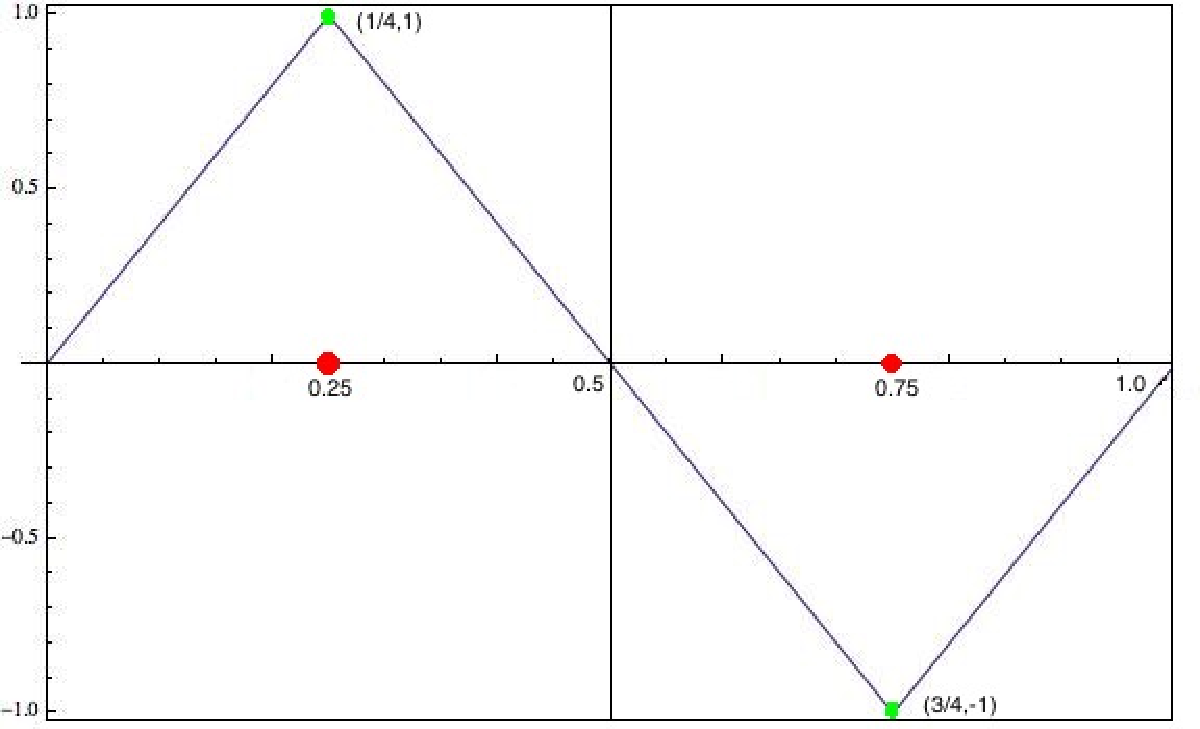}}
\caption{Pictorial representation of the image of $\left[0,1\right]$ under the quasiconformal map $L_w$; for $w=1+i/8$ (top) and $w=1$ (bottom). The Fatou coordinates of $c_0$ and $f_{c_0}^{\circ k} (c_0)$ are $1/4$ and $3/4$ respectively. For $w=1+i/8$, $L_w(1/4)=1/8+i$ and $L_w(3/4)=7/8-i$, and for $w=1$, $L_w(1/4)=1/4+i$ and $L_w(3/4)=3/4-i$.   Observe that $L_w$ commutes with $z\mapsto \overline{z}+1/2$ only when $w\in \mathbb{R}$.} 
\label{piecewise_linear} 
\end{figure} 
 
\begin{proof}
We construct a larger class of deformations which strictly contains the deformations constructed in \cite[Theorem 3.2]{MNS}. Choose the attracting Fatou coordinate $\Phi_0 : z \mapsto \zeta$ (at the parabolic point on the boundary of the unique Fatou component of $f_{c_0}$ that contains $c_0$) such that the unique geodesic invariant under the antiholomorphic dynamics maps to the real line and the critical value $c_0$ has Ecalle coordinates $\frac{1}{4}$ (this is possible since $c_0 = c(0)$). The polynomial $f_{c_0}^{\circ 2} = P_{\overline{c_0},c_0}$ has two critical values $c_0$ and $\overline{c_0}^d + c_0$. The map $f_{c_0}^{\circ k-1} = P_{\overline{c_0},c_0}^{\circ \frac{k-1}{2}}$ sends $\overline{c_0}^d + c_0$ to the Fatou component containing $c_0$. The Ecalle coordinates of $P_{\overline{c_0},c_0}^{\circ \frac{k-1}{2}}(\overline{c_0}^d + c_0) = f_{c_0}^{\circ k} (c_0)$ is $\frac{3}{4}$. We will simultaneously change the Ecalle coordinates of $c_0$ and $f_{c_0}^{\circ k} (c_0)$ in a controlled way, so that each perturbation gives a conformally different polynomial.

Setting $\zeta=x+iy$, we can change the conformal structure within the Ecalle cylinder by the quasi-conformal homeomorphism of $\left[ 0 , 1 \right] \times \mathbb{R}$ (compare Figure \ref{piecewise_linear}):

$$
L_w : \zeta \mapsto \left\{\begin{array}{ll}
                    \zeta +4i w x & \mbox{if}\ 0\leq x\leq 1/4 \\
                    \zeta + 2i w (1-2x) & \mbox{if}\ 1/4\leq x\leq 1/2 \\ 
                    \zeta - 2i w (2x-1) & \mbox{if}\ 1/2\leq x\leq 3/4 \\ 
                    \zeta -4 i w(1-x) & \mbox{if}\ 3/4\leq x\leq 1 
                      \end{array}\right. 
$$
for $w\in S$

Translating the map $L_w$ by positive integers, we get a q.c. homeomorphism of a right-half plane that commutes with the translation $z \mapsto z+1$. By the coordinate change $z\mapsto \zeta$, we can transport this Beltrami 
form (defined by this quasiconformal homeomorphism) into all the attracting petals, and it is forward invariant under $f_{c_0}^{\circ 2} = P_{\overline{c_0},c_0}$. It 
is easy to make it backward invariant by pulling it back along the 
dynamics. Extending it by the zero Beltrami form outside of  the entire 
parabolic basin, we obtain an $P_{\overline{c_0},c_0}$-invariant Beltrami form. Using the Measurable 
Riemann Mapping Theorem with parameters, we obtain a qc-map $\phi_w$ integrating this 
Beltrami form. Furthermore, if we normalize $\phi_w$ such that the conjugated map $\phi_w \circ P_{\overline{c_0},c_0} \circ \phi_w^{-1}$ is always a monic and centered polynomial, then the coefficients of this newly obtained polynomial will depend holomorphically on $w$ (since the Beltrami form depends complex analytically on $w$).

We need to check that this new polynomial belongs to our family $\mathcal{F}_d$. But this readily follows from Lemma \ref{characterization} and the remark thereafter. Therefore, this new polynomial must be of the form $P_{a(w),b(w)} = (z^d+a(w))^d+b(w)$. Since the coefficients of $P_{a(w),b(w)}$ depend holomorphically on $w$, the maps $w \mapsto a(w)$ and $w \mapsto b(w)$ are holomorphic. Its Fatou coordinate is given by $\Phi_w = L_w\circ\Phi_0\circ\phi_w^{-1}$. Note that the Ecalle coordinates of the two representatives of the two critical orbits of $\phi_w \circ P_{\overline{c_0},c_0} \circ \phi_w^{-1}$ are $L_w(\frac{1}{4}) = \frac{1}{4}+iw$ and $L_w(\frac{3}{4}) = \frac{3}{4}-iw$.

Thus, we obtain a complex analytic map $F : S \rightarrow \mathbb{C}^2 , w \mapsto (a(w) , b(w))$ (compare Figure \ref{smooth}). For real values of $w$, the map $L_w$ commutes with the map $z\mapsto \overline{z}+1/2$, and hence, the corresponding Beltrami form is invariant under the antiholomorphic polynomial $f_{c_0}$ (recall that $\Phi_0$ conjugates $f_{c_0}^{\circ k}$ to $z\mapsto \overline{z}+1/2$ in the attracting petal). Therefore, $\phi_w \circ f_{c_0} \circ \phi_w^{-1}$ is again a unicritical antiholomorphic polynomial in the family $\lbrace f_c\rbrace_{c\in \mathbb{C}}$. In particular, for $w \in \mathbb{R}$, the critical Ecalle height of $\phi_w \circ f_{c_0} \circ \phi_w^{-1}$ is $w$. Therefore, $F(w) = (\overline{c(w)},c(w) )\ \forall\ w\ \in\ \mathbb{R}.$

By construction, all the $P_{a(w),b(w)}$ are quasiconformally conjugate to $P_{\overline{c_0},c_0}$, and hence to each other. We now show that they are all conformally distinct. Define the \emph{Fatou vector} of $P_{a(w),b(w)}$ to be the quantity $\Phi_w ( P_{a(w), b(w)}^{\circ \frac{k-1}{2}}(a(w)^d+b(w))) - \Phi_w(b(w))$, where $\Phi_w$ is the Fatou coordinate of $P_{a(w),b(w)}$. Since the Fatou coordinate is unique up to addition by a complex constant, the Fatou vector defined above is a conformal conjugacy invariant. The Fatou vector of $P_{a(w),b(w)}$ is given by $\frac{1}{2} - 2iw$, which is different for different values of $w$. Hence, the polynomials $P_{a(w),b(w)}$ are conformally nonequivalent. In particular, the map $F$ is injective.

It remains to show that for any $w \in S$, $\left( a(w),b(w) \right) \in P_k$. It is easy to see that each $P_{a(w),b(w)}$ has a parabolic cycle of exact period $k$ (both critical orbits are contained in the Fatou set and converge to a $k$-periodic orbit sitting on the Julia set). Also, the first return map $P_{a(w),b(w)}^{\circ k}$ leaves the unique petal attached to every parabolic point invariant; so the multiplier of the parabolic cycle must be $1$. Hence, $\left( a(w),b(w) \right) \in P_k$.
\end{proof}

\begin{figure}[ht!]
\begin{center}
\includegraphics[scale=0.2]{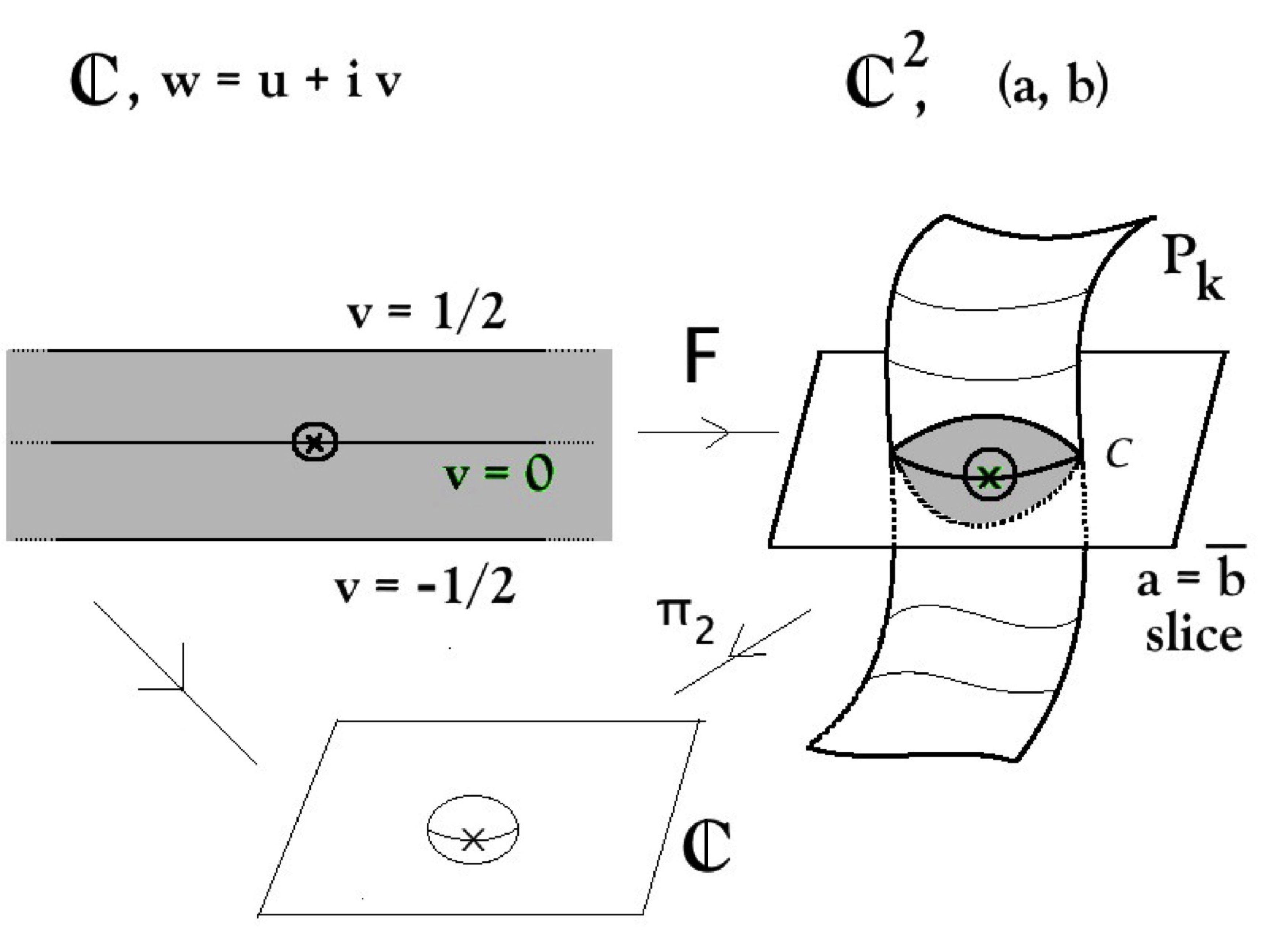}
\end{center}
\caption{$\pi_2 \circ F : w \mapsto b(w)$ is injective in a neighborhood of $\widetilde{u}$ for all but possibly finitely many $\widetilde{u} \in \mathbb{R}$.}
\label{smooth}
\end{figure} 

\begin{remark}
It was possible to construct this larger class of deformation since we worked with the polynomial $f_{c_0}^{\circ 2}$ viewing it as a member of the family $\mathcal{F}_d = \lbrace P_{a,b}(z) = \left(z^d + a \right)^d + b, \hspace{1mm} a, b \in \mathbb{C} \rbrace$. Working in the family $\lbrace f_c\rbrace_{c\in \mathbb{C}}$ of unicritical antiholomorphic polynomials would have only allowed us to construct the real one-dimensional parabolic arcs, as was done in \cite[Theorem 3.2]{MNS}. Indeed, the Beltrami form constructed in the previous lemma is invariant under $f_{c_0}^{\circ 2}$ for all $w\in S = \lbrace w = u + iv \in \mathbb{C} : \vert v \vert < \frac{1}{4} \rbrace$, but it is invariant under $f_{c_0}$ only when $w$ is real. This is because we obtained the quasiconformal deformations via the map $L_w$, which commutes with $z\mapsto \overline{z}+1/2$ only when $w \in \mathbb{R}$.
\end{remark}

\begin{remark}
Various dynamically defined quantities, e.g. the fixed point index, the Ecalle-Voronin coefficients of the parabolic cycle of $P_{a(w),b(w)}$ depend holomorphically on $w\in S$. 
\end{remark}

Recall that the set of singular points of an algebraic curve $\{(a,b)\in\mathbb{C}^2: P(a,b)=0\}$ (where $P$ is a complex polynomial in two variables $a$ and $b$) is defined as
\begin{equation*}
\{(a,b)\in\mathbb{C}^2: P(a,b)=\frac{\partial P}{\partial a}(a,b)=\frac{\partial P}{\partial b}(a,b)=0\}.
\end{equation*}
It is worth mentioning that (by the implicit function theorem) an algebraic curve is locally a manifold near its non-singular points. It is a well-known fact that an affine algebraic curve has at most finitely many singular points.

\begin{proof}[Proof of Theorem \ref{almost non-singular}]
Part (i) follows from Lemma \ref{qcdef}. The required map $\phi$ is given by $b$.

We will show that if $(\overline{c(\widetilde{u})} , c(\widetilde{u}))$ is a non-singular point of the algebraic curve $P_k$, then the map $\mathbb{R} \ni h \mapsto c(h)$ has a non-vanishing derivative at $\widetilde{u}$. Since any affine algebraic curve has at most finitely many singular points, this will complete the proof of the theorem. 

By the definition of non-singularity, one of the partial derivatives $\frac{\partial \Psi}{\partial a}$ or $\frac{\partial \Psi}{\partial b}$ is non-zero at $(\overline{c(\widetilde{u})} , c(\widetilde{u}))$. Let $\frac{\partial \Psi}{\partial a} (\overline{c(\widetilde{u})} , c(\widetilde{u}))\neq 0.$ Then there exists $\epsilon_1 , \epsilon_2 > 0$ and a holomorphic map $g : B(c(\widetilde{u}),\epsilon_1) \rightarrow B(\overline{c(\widetilde{u})},\epsilon_2)$ such that $g(c(\widetilde{u})) = \overline{c(\widetilde{u})}$ and $\Psi(a,b) = 0$ for some $(a,b) \in B(\overline{c(\widetilde{u})},\epsilon_1) \times B(c(\widetilde{u}),\epsilon_2)$ if and only if $a=g(b)$. Therefore, the projection map $\pi_2 : (a,b) \mapsto b$ is an injective holomorphic map on an open neighborhood $U \subset P_k$ (in the subspace topology) of $(\overline{c(\widetilde{u})},c(\widetilde{u}))$. It follows from Lemma \ref{qcdef} that the map $\pi_2 \circ F: w \mapsto b(w)$ is an injective holomorphic map on an open neighborhood $B(\widetilde{u},\epsilon_3) \subset \mathbb{C}$. Hence, it has a non-vanishing derivative at $\widetilde{u}$, i.e. $b'(\widetilde{u}) \neq 0$. 

Writing $b(w) = b(u + iv) = b_1(u+iv) + i b_2(u+iv)$ and $c(u) = c_1(u) + i c_2(u)$, we see that 
\begin{align*}
b'(\widetilde{u})
&= \frac{\partial b_1}{\partial u}(\widetilde{u}) + i \frac{\partial b_2}{\partial u}(\widetilde{u})
\\
&= (c_1)'(\widetilde{u}) + i (c_2)'(\widetilde{u}) \;. \hspace{4mm} [\text{For real} \hspace{1mm} u, b(u) = c(u)] 
\end{align*} 

It follows that $\left(  (c_1)'(\widetilde{u}) , (c_2)'(\widetilde{u}) \right) \neq (0,0)$, i.e. $c'(\widetilde{u}) \neq 0$.
\end{proof}

\begin{remark} 
One can easily compute that the algebraic curve $P_1$ of the family $\mathcal{F}_2$ is non-singular at each point of the parabolic arcs of period $1$ of the tricorn. Hence, the critical Ecalle height parametrization is indeed non-singular for the period $1$ parabolic arcs of the tricorn. However, we conjecture that the critical Ecalle height parametrization of any parabolic arc of $\mathcal{M}_d^*$ is non-singular everywhere.
\end{remark}

\section{Real-Analyticity of Hausdorff Dimension}\label{real-anal}

In this section, we prove real-analyticity of Hausdorff dimension of the Julia sets along the parabolic arcs by applying certain real-analyticity results from \cite{SU} to holomorphic family of parabolic maps constructed in the previous section.

Real-analyticity of Hausdorff dimension of Julia sets of hyperbolic maps is well-known due to the work of Ruelle and Bowen on thermodynamic formalism \cite{Ru,Zi}. The classical theory of thermodynamic formalism was developed (or at least works best) primarily for expanding maps. Since such maps admit Markov partitions, their dynamical properties can be conveniently studied by looking at the associated sub-shifts of finite type. In this setting, the largest eigenvalue of the so-called Ruelle operator is intimately connected with the Hausdorff dimension of the limit set. Moreover, a crucial spectral gap property of the largest eigenvalue and the fact that expanding maps are structurally stable (i.e. the qualitative behavior of the dynamics remains stable under small perturbations) help one show that Hausdorff dimension is a real-analytic function of the parameter.

Recall that a rational map is called parabolic if it has at least one parabolic cycle and every critical point of the map lies in the Fatou set (i.e. is attracted to an attracting or a parabolic cycle). Parabolic maps have a certain weak expansion property that makes it possible to use tools from thermodynamic formalism to investigate the finer fractal properties of their Julia sets. However, in the absence of Markov partitions, such a study requires much heavier machinery. We do not want to delve deep into this topic here, rather we would limit ourselves to describing how the main theorem of \cite{SU} applies to our context, as well as pointing out the salient differences between the hyperbolic and parabolic setting.  

The following concept of radial Julia sets was introduced by Urba{\'n}ski and McMullen \cite{U,Mc}, and its importance stems from the fact that the dynamics at the radial points of a Julia set have a strong expansion property.

\begin{definition}[Radial Julia Set]
A point $z \in J(f)$ is called a radial point if there exists $\delta > 0$ and an infinite sequence of positive integers $\lbrace n_k \rbrace$ such that there exists a univalent inverse branch of $f^{n_k}$ defined on $B(f^{n_k}(z), \delta)$ sending $f^{n_k}(z)$ to $z$ for all $k$. The set of all radial points of $J(f)$ is called the radial Julia set and is denoted as $J_r(f)$. Equivalently, the radial Julia set can be defined as the set of points in $J(f)$ whose $\omega$-limit set non-trivially intersects the complement of the post-critical closure.
\end{definition}

For a radial point $z$, there exists a sequence of iterates $\lbrace f^{\circ n_k}(z) \rbrace$ accumulating at a point $w$ outside the post-critical closure and hence, there exists a sequence of univalent inverse branches of $f^{\circ n_k}$ defined on some $B(w, \epsilon)$ sending $f^{\circ n_k}(z)$ to $z$ for all $k$. Such a sequence necessarily forms a normal family and any limit function must be a constant map (compare \cite[Theorem 9.2.1, Lemma 9.2.2]{B1}). This shows that the sequence of univalent inverse branches of $f^{\circ n_k}$ are eventually contracting; in other words, $\displaystyle \lim_{n_k \rightarrow \infty} (f^{\circ n_k})'(z) = \infty$. 

For parabolic rational maps, one has a rather simple but useful description of the radial Julia set. The following proposition was proved in \cite{DU}, we include the proof largely for the sake of completeness. 

\begin{lemma}\label{radial parabolic}
Let $f$ be a parabolic rational map and let $\Omega$ be the set of all parabolic periodic points of $f$. Then, $J_r(f) = J(f) \setminus \displaystyle  \bigcup_{i=0}^{\infty} f^{-i} (\Omega)$. In particular, $\mathrm{HD}(J_r(f)) = \mathrm{HD}(J(f))$.
\end{lemma}
\begin{proof}
For a parabolic rational map, the post-critical closure intersects the Julia set precisely in $\Omega$, which is a finite set consisting of (parabolic) periodic points. Starting with any point of $\displaystyle  \bigcup_{i=0}^{\infty} f^{-i} (\Omega)$, one eventually lands on $\Omega$ under the dynamics and the existence of infinitely many post-critical points in every neighborhood of $\Omega$ obstructs the existence of infinitely many univalent inverse branches. It follows that $\displaystyle  \bigcup_{i=0}^{\infty} f^{-i} (\Omega) \subset J(f) \setminus J_r(f)$.

It remains to prove the reverse inclusion. By passing to an iterate, one can assume that $f^{\circ q} (z_j) = z_j \hspace{1mm} \forall z_j \in \Omega$. By the description of the local dynamics near a parabolic point \cite[\S 10]{M1new}, there is a neighborhood $B(\Omega, \theta^{\prime})$ of the parabolic points that is contained in the union of the attracting and repelling petals. By continuity, there exists a $0< \theta < \theta^{\prime}$ such that $f^{\circ q}(B(z_i, \theta)) \cap B(z_j, \theta) = \emptyset$ for $z_i \neq z_j$ and $z_i, z_j \in \Omega$. We claim that for any $z\in J(f) \setminus \displaystyle  \bigcup_{i=0}^{\infty} f^{-i} (\Omega)$, there exists a sequence of positive integers $\lbrace n_k \rbrace$ such that the sequence $\lbrace f^{\circ qn_k}(z) \rbrace$ lies outside $B(\Omega, \theta)$. Otherwise, there would exist an $n_0 \in \mathbb{N}$ such that $f^{\circ qn}(z) \in B(\Omega, \theta) \hspace{1mm} \forall n > n_0$. Then, there exists some $z_j \in \Omega$ such that $f^{\circ qn}(z) \in B(z_j, \theta) \hspace{1mm} \forall n > n_0$. But since each $f^{\circ qn}(z)$ belongs to the repelling petal, it follows that $f^{\circ qn_0}(z) = z_j$, a contradiction to the assumption that $z \notin \bigcup_{i=0}^{\infty} f^{-i} (\Omega)$. Finally, observe that any point in $J(f) \setminus B(\Omega, \theta)$ has a small neighborhood disjoint from the post-critical set. Since $J(f) \setminus B(\Omega, \theta)$ is a compact metric space, there exists a $\delta >0$ such that the neighborhood $B(w, \delta)$ of any point $w$ of $J(f) \setminus B(\Omega, \theta)$ is disjoint from the post-critical set. Therefore, the balls $B(f^{\circ qn_k}(z), \delta)$ are disjoint from the post-critical closure and hence, there exists a univalent inverse branch of $f^{\circ qn_k}$ defined on $B(f^{\circ qn_k}(z), \delta)$ sending $f^{\circ qn_k}(z)$ to $z$ for all $k$. This proves that $J(f) \setminus \displaystyle  \bigcup_{i=0}^{\infty} f^{-i} (\Omega) \subset J_r(f)$.

The final assertion directly follows as the set $\bigcup_{i=0}^{\infty} f^{-i} (\Omega)$ is countable.
\end{proof}

We should emphasize that so far as Hausdorff dimension is concerned, the previous lemma guarantees that we do not lose anything by restricting our attention to the radial Julia set. 

Conformal measures have played a crucial role in the study of the dimension-theoretic properties of rational maps. It is worth mentioning in this regard that conformal measures satisfy a weak form of Ahlfors-regularity at the radial points of a Julia set. More precisely, by an immediate application of Koebe's distortion theorem and the expansion property at the radial points discussed above, one obtains that for every point $z$ in $J_r(f)$, there exists a sequence of radii $\displaystyle \lbrace r_k(z) \rbrace_{k=1}^{\infty}$ converging to $0$ ($r_k(z) \approx \delta \vert f^{\circ n_k})'(z) \vert^{-1}$) such that if $m$ is a $t$-conformal measure for $f$, then
\begin{align*}
C^{-1} < \frac{m(B(z, r_k(z))}{r_k(z)^t} < C
\end{align*}
for some constant $C$ depending on $\delta$ and $m$.

Before we state the main technical theorem from \cite{SU}, we need a couple of more definitions and facts from thermodynamic formalism. We will assume familiarity with the basic definitions and properties of topological pressure \cite{W}, \cite[\S 9]{W1}. The behavior of the pressure function $t \mapsto P(t)= P(f \vert_{J(f)}, -t \log\vert f' \vert), t \in \mathbb{R}$ (where $P(f \vert_{J(f)}, -t \log\vert f' \vert), t \in \mathbb{R}$ is the topological pressure) has been extensively studied by many people in the context of rational and transcendental maps. For hyperbolic rational maps, the pressure function is strictly decreasing and vanishes at a unique real number. The following result discusses the corresponding situation for parabolic maps.

\begin{theorem}\cite{DU}
\label{pressure_parabolic}
For a parabolic rational map $f$,
\begin{enumerate}
\item The function $t \mapsto P(t)$ is continuous, non-increasing and non-negative.
\item $\exists s > 0$ such that,
\begin{enumerate}
\item  \label{need} $P(t) > 0$ for $t \in \left[0, s \right)$,
\item $P(t) = 0$ for $t \in \left[s, \infty \right)$,
\item $P\vert_{\left[0, s \right]}$ is injective.
\end{enumerate} 
\end{enumerate}
\end{theorem}

The Hausdorff dimension of the Julia set of a hyperbolic map is equal to the unique zero of the associated pressure function. The next theorem relates the Hausdorff dimension of the Julia set to the the minimal zero of the pressure function and the minimum exponent of $t$-conformal measures for parabolic rational maps. 

\begin{theorem}\cite{DU}
\label{HD_Bowen}
For a parabolic rational map $f$, the following holds:
\begin{align*}
\mathrm{HD}(J(f)) &= \mathrm{inf} \lbrace t \in \mathbb{R} : \exists \hspace{1mm} t-\mathrm{conformal \hspace{1mm} measure \hspace{1mm} for} \hspace{1mm} f\vert_{J(f)} \rbrace \\
&= \mathrm{inf} \lbrace t \in \mathbb{R} : P(t) = 0 \rbrace
\end{align*}
\end{theorem}

A more elaborate account of these lines of ideas can be found in the expository article of Urba{\'n}ski \cite{DU,U}. 

\begin{definition}
A meromorphic function $f: \mathbb{C} \rightarrow \hat{\mathbb{C}}$ is called tame if its post-singular set does not contain its Julia set. 
\end{definition}

Clearly, parabolic polynomials are tame. For tame rational maps, there exist nice sets \cite{JRL,ND} giving rise to conformal iterated function systems with the property that the Hausdorff dimension of the radial Julia set is equal to the common value of the Hausdorff dimensions of the limit sets of all the iterated function systems induced by all nice sets (compare \cite[\S 2,3]{SU}). One can define the pressure function for these iterated function systems (induced by the nice sets) and the system $S$ is called strongly $N$-regular if there is $t \geq 0$ such that $0 < P_S(t) < +\infty$ and if there is $t \geq 0$ such that $P_S(t)=0$. The fact that the conformal iterated function systems arising from parabolic rational maps satisfy this property, can be proved as in Theorem \ref{pressure_parabolic}.

We are now prepared to state the main result from \cite{SU} which is at the technical heart of our real-analyticity result.

\begin{theorem}\cite[Theorem 1.1]{SU}\label{motor}
Assume that a tame meromorphic function $f : \mathbb{C} \rightarrow \hat{\mathbb{C}}$ is strongly $N$- regular. Let $\Lambda \subset \mathbb{C}^d$ be an open set and let $\lbrace f_{\lambda}\rbrace_{\lambda \in \Lambda}$ be an analytic family of meromorphic functions such that
\begin{enumerate}
\item $f_{\lambda_0} = f$ for some $\lambda_0 \in \Lambda$,

\item there exists a holomorphic motion $H : \Lambda \times J_{\lambda_0} \rightarrow \mathbb{C}$ such that each map $H_{\lambda}$ is a
topological conjugacy between $f_{\lambda_0}$ and $f_{\lambda}$ on $J_{\lambda_0}$.
\end{enumerate} 
Then the map $\Lambda \ni \lambda \mapsto \mathrm{HD}(J_r(f_{\lambda}))$ is real-analytic on some neighborhood of $\lambda_0$.
\end{theorem}

\begin{proof}[Proof of Theorem \ref{real-anal HD}]
Let $\mathcal{C}$ be a parabolic arc and $c: \mathbb{R} \rightarrow \mathcal{C}$ be its critical Ecalle height parametrization. It follows from Lemma \ref{qcdef} that there exists an injective holomorphic map $F: S \rightarrow \mathbb{C}^2, w \mapsto (a(w),b(w))$ and an analytic family of q.c. maps $\displaystyle (\phi_w)_{w \in S}$ with $\phi_w \circ P_{a(0),b(0)} \circ \phi_w^{-1} = P_{a(w),b(w)}$. Setting $\Lambda = S$, and $H = \phi(w,z) := \phi_w(z)$, we see that all the conditions of Theorem \ref{motor} are satisfied and hence, the function $w \mapsto \mathrm{HD}(J_r(P_{a(w),b(w)}))$ is real-analytic. Restricting the map to the reals, we conclude that the map $h\mapsto \mathrm{HD}(J_r(f_{c(h)}))$ is real-analytic. The result now follows from Lemma \ref{radial parabolic}.
\end{proof}  
\begin{remark}
Parabolic curves arise naturally in the study of the parameter spaces of higher degree polynomials and the techniques used in this article can be generalized to prove corresponding statements about the real-analyticity of Hausdorff dimension of the Julia sets on these curves, under suitable conditions. In particular, the real-analyticity of $\mathrm{HD}(J(f))$ continues to hold on regions of parameter spaces where the maps only have attracting or parabolic cycles and such that all the critical points converge to these cycles.
 
The parabolic arcs of the multicorns inherit the real-analyticity property from the connectedness loci of the sub-family $\mathcal{F}_d$ of all polynomials of degree $d^2$.
\end{remark}

\section{$\mathrm{Per}_n(1)$ of Biquadratic Polynomials}\label{per_1_1}

We observed that the parabolic arcs of period $n$ of $\mathcal{M}_d^*$ are naturally embedded in the algebraic curve $P_n$. Moreover, for any $c$ on a parabolic arc of period $n$, the polynomial $P_{\overline{c},c}$ is structurally stable along the curve $P_n$.

The curve $P_n$ is customarily referred to as $\mathrm{Per}_n(1)$ \cite{M3} (since the corresponding maps have a $n$-periodic orbit of multiplier $1$). The topology of these curves plays an important role in the understanding of the parameter spaces of the maps under consideration. In order to analyze the types of singularities of $\mathrm{Per}_n(1)$ and their dynamical meaning, we shall have need for some general notions about singularities of holomorphic function germs.

\subsection{Degenerate Singularities, Morsification, and Milnor Number}

A holomorphic function germ ${\displaystyle f:(\mathbb {C} ^{n},0)\to (\mathbb {C} ,0)}$ is said to be \emph{singular} at a point ${\displaystyle z_{0}\in \mathbb {C} ^{n}}$ if the first order partial derivatives ${\displaystyle \partial f/\partial z_{1},\ldots ,\partial f/\partial z_{n}}$ are all zero at ${\displaystyle z=z_{0}}$. In this subsection, we will only be concerned with isolated singularities; i.e. those singular points $z_0$ that have a small neighborhood ${\displaystyle U\subset \mathbb {C} ^{n}}$ such that ${\displaystyle z_{0}}$ is the only singular point of $f$ in $U$. We say that a point $z_0$ is a \emph{degenerate} singular point of $f$ if ${\displaystyle z_{0}}$ is a singular point and the Hessian matrix of all second order partial derivatives has zero determinant at ${\displaystyle z_{0}}$; i.e.
\begin{equation*}
{\displaystyle \left.\det \left({\frac {\partial ^{2}f}{\partial z_{i}\partial z_{j}}}\right)_{i,j=1,2,\cdots,n}\right\vert_{z=z_0}=0.} \end{equation*}

Otherwise, the singularity is called non-degenerate. Note that non-degenerate singularities are the simplest kind of singularities where the analytic set germ given by the vanishing of $f$ locally looks like the transverse intersection of two non-singular branches. In other words, the analytic set has two distinct tangent planes at a non-degenerate singularity. It is therefore desirable to think of a degenerate singularity of $f$ as the merger of several non-degenerate singularities. Intuitively, if one  perturbs $f$ suitably, then an isolated degenerate singularity of $f$ splits up into a number of non-degenerate isolated singularities. This number, which we will formally define below, measures the complexity of a singularity.

\begin{definition}[Morsification]
A $1$-parameter family of deformations $\{f_t\}$ of a holomorphic function germ $f$ with an isolated singularity is called a \emph{morsification} if for (small) $t\neq 0$ all singularities of $f_t$ are non-degenerate.
\end{definition}

Any function with an isolated singularity admits a morsification (compare \cite[Proposition 6.5.4]{W2}. This leads to the following definition of the Milnor number.

\begin{definition}[Milnor Number]
Let $\{f_t\}$ be a morsification of $f_0= f$, which has an isolated singularity at $0$. Then for small $t$, the total number of non-degenerate singularities of $f_t$ near $0$ is called the \emph{Milnor number} of $f$ at (the isolated singularity) $0$. It is denoted by $\mu(f)$.

Equivalently, one can define the Milnor number algebraically as follows. Let ${\displaystyle {\mathcal {O}}}$ be the ring of holomorphic function germs ${\displaystyle (\mathbb {C} ^{n},0)\to (\mathbb {C} ,0)}$, and ${\displaystyle J_{f}}$ be the Jacobian ideal of $f$:
\begin{equation*}
{\displaystyle J_{f}:=\left\langle {\frac {\partial f}{\partial z_{i}}}:1\leq i\leq n\right\rangle .} 
\end{equation*}
The local algebra of $f$ is then given by:
\begin{equation*}
{\displaystyle {\mathcal {A}}_{f}:={\mathcal {O}}/J_{f}.} 
\end{equation*}
The Milnor number is then equal to the dimension of ${\mathcal {A}}_{f}$ as a complex vector space:
\begin{equation*}
{\displaystyle \mu (f)=\dim _{\mathbb {C} }{\mathcal {A}}_{f}\ .}
\end{equation*}
\end{definition}

It may be instructive to compute the Milnor number of a couple of simple functions germs using the algebraic description.

\textbf{Example 1).} The function $f: \mathbb{C}^2\to\mathbb{C}, f(x,y)=x^2+y^2$ has a singularity at $(0,0)$ with non-vanishing Hessian. Since $\partial f/\partial x=2x$, and $\partial f/\partial y=2y$, the Jacobian ideal $\mathcal{J}_f$ is $\langle 2x,2y\rangle=\langle x,y\rangle$. Hence its local algebra $\mathcal{A}_f$ is given by $\mathcal{O}/\langle x,y\rangle\cong\langle 1\rangle$, which has dimension $1$. Therefore, the Milnor number is $\mu(f)=1$. In fact, any function germ with a non-degenerate singularity has Milnor number $1$.

\textbf{Example 2).} The function $g: \mathbb{C}^2\to\mathbb{C}, g(x,y)=x^2+y^3$ has a singularity at $(0,0)$ with vanishing Hessian. Since $\partial g/\partial x=2x$, and $\partial g/\partial y=3y^2$, the Jacobian ideal $\mathcal{J}_g$ is $\langle 2x,3y^2\rangle=\langle x,y^2\rangle$. Hence its local algebra $\mathcal{A}_g$ is given by $\mathcal{O}/\langle x,y^2\rangle\cong\langle 1,y\rangle$, which has dimension $2$. Therefore, the Milnor number is $\mu(g)=2$. It is not hard to see that under small perturbations, the degenerate singularity of $g$ splits into two distinct non-degenerate singularities.

We refer the readers to \cite[\S 7]{M5}, \cite[\S 6]{W2} for a more comprehensive account on these concepts.

\subsection{Dynamically Defined Morsifications}

With these general tools at our disposal, we now return to the study of the singularities of the curves $\mathrm{Per}_n(1)$. 

For $n=1$ and $d=2$, the curve $\mathrm{Per}_1(1)$ of the family of polynomials $\mathcal{F}_2 = \lbrace P_{a,b}(z) = \left(z^2 + a \right)^2 + b,\ a, b \in \mathbb{C}\rbrace$ has a very simple description. One easily computes that 

\begin{align*}
\mathrm{Per}_1(1)
&= \lbrace (a, b) \in \mathbb{C}^2 : \mathrm{disc}_z \left( P_{a,b}(z) - z \right) = 0 \rbrace
\\
&= \lbrace (a, b) \in \mathbb{C}^2 : h_1(a,b):= 256a^3 + 288ab + 256a^2b^2 + 256b^3 - 27=0\rbrace
\end{align*}

The set of singular points of $\mathrm{Per}_1(1)$ is $\{ (-\frac{3}{4}$, $-\frac{3}{4})$, $(-\frac{3}{4}\omega$, $-\frac{3}{4}\omega^2),$ $(-\frac{3}{4}\omega^2$, $-\frac{3}{4}\omega)\}$, where $\omega$ is a primitive third root of unity. Note that these points correspond precisely to the cusps at the ends of the parabolic arcs of period $1$ of the tricorn. In fact, a simple calculation shows that 
\begin{center}
$\frac{\partial h_1}{\partial a}\vert_{\left(-\frac{3}{4},-\frac{3}{4}\right)}=0,\ \frac{\partial h_1}{\partial b}\vert_{\left(-\frac{3}{4},-\frac{3}{4}\right)}=0,$
\\ $\frac{\partial^2 h_1}{\partial a^2}\vert_{\left(-\frac{3}{4},-\frac{3}{4}\right)}\frac{\partial^2 h_1}{\partial b^2}\vert_{\left(-\frac{3}{4},-\frac{3}{4}\right)}- \left(\frac{\partial^2 h_1}{\partial a \partial b}\vert_{\left(-\frac{3}{4},-\frac{3}{4}\right)}\right)^2=0$.
\end{center}

The same is true for the other singular points. In particular, each singular point of $\mathrm{Per}_1(1)$ is an ordinary cusp point (i.e. a cusp of the form $x^2+y^3=0$ at $(0,0)$). By the classical degree-genus formula for singular curves (compare \cite[Theorem 7.37]{K1}), it follows that the genus of $\mathrm{Per}_1(1)$ is $0$. Hence, after desingularization (and compactification), $\mathrm{Per}_1(1)$ of the family $\mathcal{F}_2$ is the Riemann sphere $\hat{\mathbb{C}}$. 

It will be useful to consider a family of dynamically defined deformations of the function $h_1$ above so that the perturbed functions only have non-singular singularities. To be precise, we will look at the curve $\mathrm{Per}_1(r)$, which consists of parameters $(a,b)$ such that $P_{a,b}$ has a fixed point of multiplier $r$. In other words, $(a,b)$ lies on $\mathrm{Per}_1(r)$ if the polynomials $(P_{a,b}(z) - z)$ and $(P_{a,b})'(z) - r)$ have a common root. This allows us to define the algebraic curve $\mathrm{Per}_1(r)$ in terms of (sub-)resultants as
\begin{align*}
\mathrm{Per}_1(r)
&:= \lbrace (a, b) \in \mathbb{C}^2 : \textrm{sRes}_0 \left( P_{a,b}(z) - z, (P_{a,b})'(z) - r \right) = 0 \rbrace
\\
&= \lbrace (a, b) \in \mathbb{C}^2 : h_r(a,b):= 256 a^3 + 256 a b + 256 a^2 b^2 + 256 b^3 
\\
&+32ab(2r-r^2)+(r^4-12r^3+ 48 r^2 -64r)=0\rbrace
\end{align*}
 
Let us discuss the singularities of $h_r$ (with $r\in (1-\epsilon,1)$, $\epsilon>0$ sufficiently small) near $\left(-\frac{3}{4},-\frac{3}{4}\right)$. Each $h_r$ has two non-degenerate critical points close to $\left(-\frac{3}{4},-\frac{3}{4}\right)$. Hence, the deformation $\{h_r\}$ provides a morsification of the degenerate singularity $\left(-\frac{3}{4},-\frac{3}{4}\right)$ of $h_1$ (compare Figure \ref{morsification}). Observe that this is in consonance with the fact that $\left(-\frac{3}{4},-\frac{3}{4}\right)$ is a degenerate singularity of $h_1$ with Milnor number $2$.

One of these two critical points $(\overline{\alpha_r}, \alpha_r)=(\frac{r-4}{4}, \frac{r-4}{4})$ of $h_r$ lies on $\mathrm{Per}_1(r)$. In fact, $(\overline{\alpha_r}, \alpha_r)$ belongs to a period $1$ hyperbolic component (bifurcating from the principal hyperbolic component at $\left(-\frac{3}{4},-\frac{3}{4}\right)$) of $\mathcal{F}_2$. Therefore, $h_r(\overline{\alpha_r}, \alpha_r)=0$, and its Milnor number is $\mu(h_r,(\overline{\alpha_r}, \alpha_r))=1$. One can dynamically explain the existence of the singular point $(\overline{\alpha_r}, \alpha_r)$ of  $\mathrm{Per}_1(r)$ as follows: the map $P_{\overline{\alpha_r}, \alpha_r}$ has two distinct fixed points of multiplier $r$, and hence $(\overline{\alpha_r}, \alpha_r)$ is a point of transverse intersection of two smooth local branches (each defined by the condition that $P_{a,b}$ has a fixed point of multiplier $r$) of $\mathrm{Per}_1(r)$.

\begin{figure}[ht!]
\begin{center}
\includegraphics[scale=0.75]{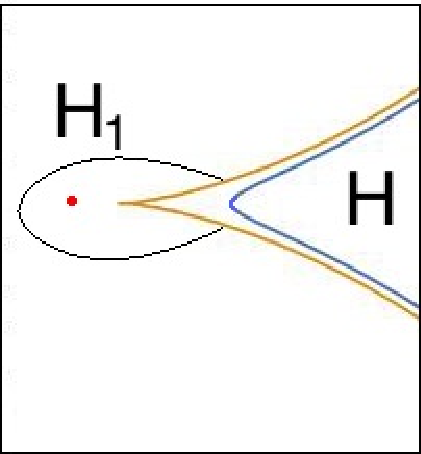}
\end{center}
\caption{The outer yellow curve indicates part of $\mathrm{Per}_1(1)\cap \lbrace a=\overline{b}\rbrace$, and the inner blue curve (along with the red point) indicates part of the deformation $\mathrm{Per}_1(r)\cap \lbrace a=\overline{b}\rbrace$ for some $r\in (1-\epsilon,1)$. The cusp point $c_0$ on the yellow curve is a critical point of $h_1$, i.e. a singular point of $\mathrm{Per}_1(1)$, and the red point is a critical point of $h_r$; i.e a singular point of $\mathrm{Per}_1(r)$.}
\label{morsification}
\end{figure}
Using these information on the family $\mathrm{Per}_1(r)$, we will now show that the cusp points of period $n$ (recall Definition \ref{DefCusp}), lying on the boundaries of hyperbolic components of odd period $n$ of the tricorn, are singular points (with at least a double tangent) of $\mathrm{Per}_n(1)$ of the family $\mathcal{F}_2$. 

\begin{proposition}
Let $\mathcal{C}$ be a parabolic arc of odd period $n$ of $\mathcal{M}_2^*$, and $c_0$ be a double parabolic point (cusp point) at the end of $\mathcal{C}$. Then $(\overline{c_0},c_0)$ is a singular point (with at least a double tangent) of $\mathrm{Per}_n(1)$ of $\mathcal{F}_2$.
\end{proposition}

\begin{proof}
Recall that the set of singular points of $\mathrm{Per}_1(1)$ is $\{ (-\frac{3}{4}$, $-\frac{3}{4})$, $(-\frac{3}{4}\omega$, $-\frac{3}{4}\omega^2),$ $(-\frac{3}{4}\omega^2$, $-\frac{3}{4}\omega)\}$ (where $\omega$ is a primitive third root of unity), and these are precisely the cusp points of period $1$. We will denote the principal hyperbolic component of the family $\mathcal{F}_2$ by $H$, and the hyperbolic component (of $\mathcal{F}_2$) of period $1$ that bifurcates from $H$ at the parameter $\left(-\frac{3}{4},-\frac{3}{4}\right)$ by $H_1$. 

Let $\widetilde{H}$ be the hyperbolic component of period $n$ of $\mathcal{F}_2$ with $\mathcal{C} \subset \partial \widetilde{H}$, and $\widetilde{H}_1$ be the hyperbolic component of period $n$ of $\mathcal{F}_2$ that bifurcates from $\widetilde{H}$ at $(\overline{c_0},c_0)$. We again consider the $1$-parameter family of deformations $\mathrm{Per}_n(r)$ (the curve of parameters with an $n$-periodic cycle of multiplier $r$) of $\mathrm{Per}_n(1)$. More precisely, we define $\mathrm{Per}_n(r):=\lbrace (a,b)\in \mathbb{C}^2 : \widetilde{h}_r(a,b)=0\rbrace$, where $\widetilde{h}_r(a,b)$ is the square-free part of $\textrm{sRes}_0(P_{a,b}^{\circ n}(z)-z, (P_{a,b}^{\circ n})'(z)-r)$.

By \cite[Theorem C]{IK}, the straightening map induces biholomorphisms $\chi : \widetilde{H} \to H$ and $\chi :\widetilde{H}_1\to H_1$, and $\chi(\overline{c_0},c_0)= \left(-\frac{3}{4},-\frac{3}{4}\right)$ (this is no loss of generality as one of the maps $\chi$, $\omega \chi$ or $\omega^2 \chi$ satisfies this property). Since $r<1$, $\chi$ sends a polynomial with an attracting $n$-cycle of multiplier $r$ to a polynomial with an attracting fixed point of multiplier $r$; i.e. $\chi(\widetilde{H}_1\cap \mathrm{Per}_n(r))= H_1\cap \mathrm{Per}_1(r)$. Therefore, 
\begin{align*}
\widetilde{H}_1\cap \mathrm{Per}_n(r)
&= \chi^{-1}(H_1\cap \mathrm{Per}_1(r))
\\
&= \chi^{-1}(\lbrace (a,b)\in H_1 : h_r(a,b)=0\rbrace)
\\
&= \lbrace\chi^{-1}(a,b)\in \widetilde{H}_1 : h_r(a,b)=0\rbrace
\\
&= \lbrace (a,b)\in \widetilde{H}_1 : h_r\circ\chi(a,b)=0\rbrace .
\end{align*}

Clearly, $\chi^{-1}(\overline{\alpha_r}, \alpha_r)$ lies on the curve $\mathrm{Per}_n(r)$; i.e. $\widetilde{h}_r(\chi^{-1}(\overline{\alpha_r}, \alpha_r))=0$. We claim that $\chi^{-1}(\overline{\alpha_r}, \alpha_r)$ is a non-degenerate critical point of $\widetilde{h}_r$, for each $r\in (1-\epsilon,1)$. Recall that $(\overline{\alpha_r}, \alpha_r)\to \left(-\frac{3}{4},-\frac{3}{4}\right)$ as $r\to 1$, hence $\chi^{-1}(\overline{\alpha_r}, \alpha_r) \to \chi^{-1}\left(-\frac{3}{4},-\frac{3}{4}\right)=(\overline{c_0},c_0)$ as $r\to 1$ (by \cite[Theorem 6.3]{IM1}, the straightening map $\chi$ extends as a homeomorphism from the closure of $\widetilde{H}$ onto the closure of $H$). Since $\widetilde{h}_r$ (for $\epsilon>0$ small enough) is an arbitrarily small perturbation of $\widetilde{h}_1$, this claim implies that any small perturbation $\widetilde{h}_r$ of $\widetilde{h}_1$ has at least one non-degenerate critical point near $(\overline{c_0},c_0)$. Hence, by definition, the Milnor number $\mu(\widetilde{h}_1, (\overline{c_0},c_0))$ is at least $1$ (compare \cite[Lemma 6.5.5]{W2}); i.e. $(\overline{c_0},c_0)$ is a singular point of $\mathrm{Per}_n(1)$.

We will now prove the claim that $\chi^{-1}(\overline{\alpha_r}, \alpha_r)$ is a non-degenerate critical point of $\widetilde{h}_r$; i.e. it has Milnor number $\mu(\widetilde{h}_r,\chi^{-1}(\overline{\alpha_r}, \alpha_r))= 1$. Note that since $(\overline{\alpha_r}, \alpha_r)$ ($\in H_1$) is a non-degenerate critical point of $h_r$ and $\chi$ is a biholomorphism from $\widetilde{H}_1$ onto $H_1$, it follows (by a simple computation using chain rule) that $\chi^{-1}(\overline{\alpha_r}, \alpha_r)$ ($\in \widetilde{H}_1$) is a non-degenerate critical point of $h_r\circ\chi$. Therefore, the Milnor number $\mu(h_r\circ\chi,\chi^{-1}(\overline{\alpha_r}, \alpha_r))= 1$. It now suffices to look at the relation between $\widetilde{h}_r$ and $h_r\circ\chi$ locally near $\chi^{-1}(\overline{\alpha_r}, \alpha_r)$. This is a routine exercise in analytic geometry, we work out the details for the sake of completeness. Consider the ring $\mathcal{O}_{2,\chi^{-1}(\overline{\alpha_r}, \alpha_r)}$ of germs of holomorphic functions (in two variables) defined in some neighborhood of $\chi^{-1}(\overline{\alpha_r}, \alpha_r)$. Since $\widetilde{h}_r$ and $h_r\circ\chi$ are both square-free as elements of $\mathcal{O}_{2,\chi^{-1}(\overline{\alpha_r}, \alpha_r)}$ and their vanishing define the same analytic set germ near $\chi^{-1}(\overline{\alpha_r}, \alpha_r)$ (namely the germ of $\mathrm{Per}_n(r)$ at $\chi^{-1}(\overline{\alpha_r}, \alpha_r)$), it follows that there exists an invertible element $u \in \mathcal{O}_{2,\chi^{-1}(\overline{\alpha_r}, \alpha_r)}$ such that $\widetilde{h}_r=u(h_r\circ\chi)$ as elements of $\mathcal{O}_{2,\chi^{-1}(\overline{\alpha_r}, \alpha_r)}$. But then, their Milnor numbers are equal; i.e. $\mu(h_r\circ\chi,\chi^{-1}(\overline{\alpha_r}, \alpha_r))= 1=\mu(\widetilde{h}_r,\chi^{-1}(\overline{\alpha_r}, \alpha_r))$. 

To conclude the proof, we need to justify that $\mathrm{Per}_n(1)$ has a double tangent at $(\overline{c_0},c_0)$. A direct computation shows that the two distinct tangent lines of $\mathrm{Per}_1(r)$ at $(\overline{\alpha_r}, \alpha_r)$ tend to coincide as $r$ tends to $1$ (in fact, they both tend to $a=b$, which is a double tangent of $\mathrm{Per}_1(1)$ at $(-\frac{3}{4},\frac{3}{4})$). This property is preserved by the biholomorphism $\chi$. Hence the two distinct tangent lines of $\mathrm{Per}_n(r)$ at $\chi^{-1}(\overline{\alpha_r}, \alpha_r)$ tend to coincide as $r$ tends to $1$, and they form a double tangent line of $\mathrm{Per}_n(1)$ at $(\overline{c_0},c_0)$.
\end{proof}

On the other hand, we observed that the curve $\mathrm{Per}_1(1)$ is non-singular everywhere along the parabolic arcs of period $1$; i.e. $(dh_1)(\overline{c},c)\neq 0\ \forall c \in \mathcal{C}$. Consequently, for $r\in (1-\epsilon,1)$, the dynamically defined deformations $h_r$ have no critical points near $\mathcal{C}$. Since this property is preserved under the straightening map $\chi$, and since the curve germs $\chi^{-1}(\mathrm{Per}_1(r))$ form a $1$-parameter family of deformations of $\mathrm{Per}_n(1)$, we have the following proposition:

\begin{proposition}
Let $\mathcal{C}$ be a parabolic arc of odd period $n$ of $\mathcal{M}_2^*$, and $c: \mathbb{R}\to \mathcal{C}$ be its critical Ecalle height parametrization. Then for each $c \in \mathcal{C}$, $(\overline{c},c)$ is a non-singular point of $\mathrm{Per}_n(1)$. In particular, the critical Ecalle height parametrization of $\mathcal{C}$ has a non-vanishing derivative at all points; i.e. $c'(t)\neq 0$ for every $t\in \mathbb{R}$.
\end{proposition}

\section{Singularities of $\mathrm{Per}_n(1)$: Some Examples}\label{sing_per_n_1}
In this section, we will take a closer look at the algebraic sets $\mathrm{Per}_n(1)$ that appeared earlier in the article. The singular locus of these algebraic sets are important in understanding their topology. Indeed, the local topology of a curve near a singular point is completely determined by its equisingularity class. In particular, one can associate a link $K$ with a plane curve singularity such that two equisingular curves have isotopic links, and the curve is locally homeomorphic to the cone on $K$ \cite[Theorem 5.5.9, Lemma 5.2.1]{W2}

\begin{definition}
Let $\mathcal{F}=\lbrace f_{a_1, a_2, \cdots, a_m}\rbrace_{(a_1, a_2, \cdots, a_m) \in \mathbb{C}^m}$ be a holomorphic family of holomorphic polynomials of degree $d\geq 2$, depending algebraically on parameters $(a_1, a_2,$ $\cdots, a_m)\in \mathbb{C}^m$. The algebraic set $\mathrm{Per}_n(1)$ is defined as the set of parameters in $\mathbb{C}^m$ such that $f_{a_1, a_2, \cdots, a_m}$ has a parabolic cycle of period $n$ and multiplier $1$. In algebraic terms, 
\begin{center}
$\mathrm{Per}_n(1):=\lbrace (a_1, a_2, \cdots, a_m) \in \mathbb{C}^m: \mathrm{disc}_z(f_{a_1, a_2, \cdots, a_m}^{\circ n}(z)-z)=0\rbrace$,
\end{center}
\end{definition}

Observe that our definition does not ensure that for every $(a_1, a_2, \cdots, a_m)\in \mathrm{Per}_n(1)$, the corresponding polynomial $f_{a_1, a_2, \cdots, a_m}$ has a periodic orbit of \emph{exact} period $n$ with multiplier $1$. Indeed, our definition allows $\mathrm{Per}_n(1)$ to contain all parameters such that the corresponding polynomials possess a $n'$-periodic orbit (where $n'\vert n$) with multiplier a $n/n'$-th root of unity. However, parameters having a $n'$-periodic orbit (where $n'\vert n$) with multiplier a $n/n'$-th root of unity determine a polynomial that divides $\mathrm{disc}_z \left( f_{a_1, a_2, \cdots, a_m}^{\circ n}(z) - z \right)$, and one can factor them out from $\mathrm{disc}_z \left( f_{a_1, a_2, \cdots, a_m}^{\circ n}(z) - z \right)$ to obtain an algebraic curve consisting precisely of those maps with an $n$-cycle with multiplier $1$. However, we assure the readers that this ambiguity in defining $\mathrm{Per}_n(1)$ will not be of importance to us since we will mostly be interested in local properties (e.g. singularity) that are not affected by the additional components of $\mathrm{Per}_n(1)$.

In what follows, we will look at some more examples of $\mathrm{Per}_n(1)$ (in various families of polynomials), and will try to understand the nature of their singularities as well as the `dynamical' behavior of these singular parameters. It is not our aim to prove any precise theorem here, we only intend to investigate some natural examples and to give heuristic explanations of the phenomena, which should pave the way for a more general understanding of the topology of these algebraic sets.

1. We first consider families $\mathcal{F}$ such that $m=2$, a generic $f_{a_1, a_2}$ has two infinite critical orbits, and $\mathrm{Per}_n(1)$ is a complex curve. For any $(a_1, a_2) \in \mathrm{Per}_n(1)$, we can write $f_{a_1, a_2}^{\circ n}$ in appropriate local coordinates in a neighborhood of a parabolic periodic point as $f_{a_1, a_2}^{\circ n}(z)=z+z^{k+1}+O(\vert z\vert^{k+2})$. Such a parabolic point has $k$ petals and each petal attracts at least one infinite critical orbit. Since every map in our family has at most $2$ infinite critical orbits, it follows that $k\in \lbrace 1,2\rbrace$. Furthermore, any $f_{a_1, a_2}$ has at most two disjoint parabolic cycles (since there are at most two infinite critical orbits). We, therefore, have to consider the following three cases.

(a) (Unique parabolic cycle with $k=1$: non-singular point). Let $U$ be a sufficiently small neighborhood of $(\widetilde{a_1}, \widetilde{a_2})$ in $\mathbb{C}^2$. The double root of $f_{\widetilde{a_1}, \widetilde{a_2}}^{\circ n}-\mathrm{id}$ splits into two simple roots in $U\setminus \mathrm{Per}_n(1)$. On a double cover $B$ over $U$, ramified only over $U\cap\mathrm{Per}_n(1)$, we can follow these two simple periodic points holomorphically as functions $z_1$ and $z_2$. In $B$, $U\cap\mathrm{Per}_n(1)$ corresponds to the smooth (complex $1$-dimensional) analytic set $\lbrace z_1=z_2\rbrace$ (some care is needed here; if $(\widetilde{a},\widetilde{b})$ are the local coordinates on $B$, one needs to show that the map $(\widetilde{a},\widetilde{b})\mapsto (z_1(\widetilde{a},\widetilde{b}),z_2(\widetilde{a},\widetilde{b}))$ is a local biholomorphism). This suggests that $(\widetilde{a_1}, \widetilde{a_2})$ is a non-singular point of $\mathrm{Per}_n(1)$.

Alternatively, in a sufficiently small neighborhood $U$ of $(\widetilde{a_1}, \widetilde{a_2})$ in $\mathbb{C}^2$, we have $f_{a_1, a_2}^{\circ n}(z)-z=(z^2+2p(a_1, a_2)z+q(a_1, a_2))g_{a_1, a_2}(z)$, where $p(a_1, a_2), q(a_1, a_2)$ are holomorphic in $(a_1, a_2)$, and $g_{a_1, a_2}$ is non-vanishing. Hence, $U\cap\mathrm{Per}_n(1)=\lbrace(a_1, a_2)\in U: \mathrm{disc}_z (z^2+2p(a_1, a_2)z+q(a_1, a_2))=p(a_1, a_2)^2-q(a_1, a_2)=0\rbrace$. In a generic situation, $(p(a_1, a_2),q(a_1, a_2))$ can be taken as local coordinates on $U$, and locally near $(\widetilde{a_1}, \widetilde{a_2})$, $U\cap\mathrm{Per}_n(1)$ will  look like $u^2=v$ at $(0,0)$. Hence, $(\widetilde{a_1}, \widetilde{a_2})$ would be a non-singular point of $\mathrm{Per}_n(1)$.

(b) (Unique parabolic cycle with $k=2$: ordinary cusp). Let $U$ be a sufficiently small neighborhood of $(\widetilde{a_1}, \widetilde{a_2})$ in $\mathbb{C}^2$. The triple root of $f_{\widetilde{a_1}, \widetilde{a_2}}^{\circ n}-\mathrm{id}$ splits into three simple roots in $U\setminus \mathrm{Per}_n(1)$. On a triple cover $B$ over $U$, ramified only over $U\cap\mathrm{Per}_n(1)$, we can follow these three simple periodic points holomorphically as functions $z_1$, $z_2$ and $z_3$. In $B$, $U\cap\mathrm{Per}_n(1)$ corresponds to the complex $1$-dimensional analytic set $\lbrace z_1=z_2\rbrace\cup\lbrace z_1=z_3\rbrace\cup\lbrace z_3=z_2\rbrace$, and the parameter $(\widetilde{a_1}, \widetilde{a_2})$ corresponds to $\lbrace z_1=z_2=z_3\rbrace$. This analytic set is not regular at $\lbrace z_1=z_2=z_3\rbrace$: two of the three conditions $\lbrace z_i=z_j\rbrace$ determine the same curve, and hence, the analytic set $\lbrace z_1=z_2\rbrace\cup\lbrace z_1=z_3\rbrace\cup\lbrace z_3=z_2\rbrace$ has two coincident tangent lines (or a tangent line of multiplicity $2$) at $\lbrace z_1=z_2=z_3\rbrace$. This suggests that $(\widetilde{a_1}, \widetilde{a_2})$ is an ordinary cusp singularity of $\mathrm{Per}_n(1)$.

Alternatively, in a sufficiently small neighborhood $U$ of $(\widetilde{a_1}, \widetilde{a_2})$ in $\mathbb{C}^2$, we have $f_{a_1, a_2}^{\circ n}(z)-z=(z^3-3p(a_1, a_2)z+2q(a_1, a_2))g_{a_1, a_2}(z)$, where $p(a_1, a_2), q(a_1, a_2)$ are holomorphic in $(a_1, a_2)$, and $g_{a_1, a_2}$ is non-vanishing. Hence, $U\cap\mathrm{Per}_n(1)=\lbrace(a_1, a_2)\in U: \mathrm{disc}_z (z^3+3p(a_1, a_2)z+2q(a_1, a_2))=q(a_1, a_2)^2-p(a_1, a_2)^3=0\rbrace$. In a generic situation, $(p(a_1, a_2),q(a_1, a_2))$ can be taken as local coordinates on $U$, and  $U\cap\mathrm{Per}_n(1)$ will have a singularity of the
form $u^2=v^3$ at $(\widetilde{a_1}, \widetilde{a_2})$.

(c) (Two parabolic cycles each with $k=1$: ordinary double point). Let $U$ be a sufficiently small neighborhood of $(\widetilde{a_1}, \widetilde{a_2})$ in $\mathbb{C}^2$. Then $f_{\widetilde{a_1}, \widetilde{a_2}}^{\circ n}-\mathrm{id}$ has two (distinct) double roots. When $(\widetilde{a_1}, \widetilde{a_2})$ is perturbed in $U\setminus \mathrm{Per}_n(1)$, these two double roots split into two pairs of simple roots. On a double cover $B$ over $U$, ramified only over $U\cap\mathrm{Per}_n(1)$, we can follow these two pairs of simple periodic points as two pairs of holomorphic functions $\lbrace z_1, z_2\rbrace$ and $\lbrace z_3, z_4\rbrace$. In $B$, $U\cap\mathrm{Per}_n(1)$ corresponds to the complex $1$-dimensional analytic set $\lbrace z_1=z_2\rbrace\cup\lbrace z_3=z_4\rbrace$, and the parameter $(\widetilde{a_1}, \widetilde{a_2})$ corresponds to the point of self-intersection $\lbrace z_1=z_2\rbrace\cap\lbrace z_3=z_4\rbrace$. In other words, $U\cap\mathrm{Per}_n(1)$ corresponds to the union of two (different) branches given by $\lbrace z_1=z_2\rbrace$ and $\lbrace z_3=z_4\rbrace$, and $(\widetilde{a_1}, \widetilde{a_2})$ corresponds to the point where these two branches intersect (transversally). Therefore, $(\widetilde{a_1}, \widetilde{a_2})$ is an ordinary double point (with non-vanishing Hessian and two distinct tangent lines) of $\mathrm{Per}_n(1)$.

\textbf{E}xamples. 
i) In Subsection \ref{per_1_1}, we discussed the properties of $\mathrm{Per}_1(1)$ of the family $\lbrace (z^2+a)^2+b\rbrace_{a,b \in \mathbb{C}}$: the only singularities of this algebraic curve correspond to the parameters with double parabolic points, and each of these singularities is of the form $x^2-y^3$ at $(0,0)$. Every other point of this curve is non-singular.

ii) Let us look at the family of monic centered cubic polynomials $\mathcal{F}:=\lbrace f_{a,b}(z)= z^3 - 3az + b,\ a,b \in \mathbb{C}\rbrace$. For this family, we have
\begin{align*}
\mathrm{Per}_1(1)
&=\lbrace (a,b)\in \mathbb{C}^2: \mathrm{disc}_z(f_{a,b}(z)-z)=0\rbrace
\\
&= \lbrace (a,b)\in \mathbb{C}^2: 4 + 36 a + 108 a^2 + 108 a^3 - 27 b^2=0\rbrace
\end{align*}

Once again, the only singular point of this curve is $(-\frac{1}{3},0)$, which is the only double parabolic point of $\mathcal{F}$. In fact, $(-\frac{1}{3},0)$ is a double point of $\mathrm{Per}_1(1)$ with vanishing Hessian. To better visualize this singularity, we may consider the following dynamically defined morsification. The algebraic curve consisting of all parameters $(a,b)$ for which $f_{a,b}$ has a fixed point of multiplier $r$ is referred to as $\mathrm{Per}_1(r)$. More precisely, for $r \in \mathbb{C}$, we define
\begin{eqnarray*}
h_r(a,b) &:=& \mathrm{Res}_z(f_{a,b}(z)-z, f_{a,b}'(z)-r), \\
&=& 108 a^3 + 108 a^2  - 27 b^2 -9a(r-3)(r+1)+ r(r-3)^2,\\
\mathrm{Per}_1(r) &:=& \lbrace (a,b)\in \mathbb{C}^2: h_r(a,b)=0\rbrace ,
\end{eqnarray*}
where $\mathrm{Res}_z(g_1, g_2)$ of two polynomial $g_1$ and $g_2$ in one complex variable denotes the resultant of $g_1$ and $g_2$.

Then $h_r$ is a morsification of $h_1$ in the sense that for $r\neq 1$, $h_r$ has two non-degenerate (i.e. with non-vanishing Hessian) critical points at $\left(\frac{r-3}{6}, 0\right)$ and $\left(-\frac{r+1}{6}, 0\right)$. Let us try to understand the formation of the cusp-type singularity of $h_1$ at $(-1/3,0)$ in the following dynamical fashion. For $(a_1,b_1)=\left(\frac{r-3}{6}, 0\right)$, $f_{a_1,b_1}$ has two distinct fixed points of multiplier $r$. Hence, $\left(\frac{r-3}{6}, 0\right)$ is the point of intersection of two transversal branches of $\mathrm{Per}_1(r)$, and is a node-type singular point of $\mathrm{Per}_1(r)$ (i.e. a double point with two distinct tangent lines). At the other critical point $(a_2,b_2)=\left(-\frac{r+1}{6}, 0\right)$ of $h_r$, $f_{a_2,b_2}$ has two distinct fixed points of multiplier $(2-r)$. As $r\to 1$, these two critical points (each having Milnor number $1$) of $h_r$ coalesce to form the critical point $(-1/3,0)$ (with Milnor number $2$) of $h_1$, and all the fixed points coalesce to form a triple fixed point of $f_{-1/3,0}$. 

iii) We now consider $\mathrm{Per}_1(1)$ of the family $\lbrace (z^3+a)^3+b\rbrace_{a,b \in \mathbb{C}}$. The parameter $(a_0,b_0) \approx (0.7698+0.7698i, 0.7698-0.7698i)$ is a co-root of a hyperbolic component of period $2$ of $\mathcal{M}_3^*$. The polynomial $(z^3+a_0)^3+b_0$ has two simple parabolic fixed points, and $(a_0,b_0)$ is an ordinary double point (with non-vanishing Hessian and two distinct tangent lines) of $\mathrm{Per}_1(1)$ of this family.

The local topology of $\mathrm{Per}_n(1)$ near these singularities can be studied via the singularity link and the Milnor fibration (compare \cite{M5}). In fact, in the first two examples above, $\mathrm{Per}_n(1)$ is locally (near the singularity) a cone over the trefoil knot. In the third example, $\mathrm{Per}_n(1)$ is locally a cone over the Hopf link.

2. We now look at families with $m\geq 3$ such that a generic $f_{a_1, a_2, a_3}$ has three infinite critical orbits. For simplicity, we will restrict ourselves to the case $m=3$ (this case is already more complicated than the case where $m=2$). As above, for any $(a_1, a_2, a_3) \in \mathrm{Per}_n(1)$, we can write $f_{a_1, a_2, a_3}^{\circ n}$ in appropriate local coordinates in a neighborhood of a parabolic point as $f_{a_1, a_2, a_3}^{\circ n}(z)=z+z^{k+1}+O(\vert z\vert^{k+2})$ with $k\in \lbrace 1,2,3\rbrace$. For ease of exposition, let us work with the family $\mathcal{F}:=\lbrace f_{a,b,c}(z)= z^4 + az^2 + bz + c,\ a,b,c \in \mathbb{C}\rbrace$ of monic centered quartic polynomials. Then, 
\begin{align*}
\mathrm{Per}_1(1)
&=\lbrace (a,b,c)\in \mathbb{C}^3: \mathrm{disc}_z(z^4 + az^2 + bz + c-z)=0\rbrace
\\
&= \lbrace (a,b,c)\in \mathbb{C}^3: -27 - 4 a^3 + 108 b + 8 a^3 b 
\\
&- 162 b^2 - 4 a^3 b^2
+ 108 b^3 - 27 b^4 + 144 a c + 16 a^4 c
\\
& - 288 a b c + 144 a b^2 c - 128 a^2 c^2 + 256 c^3=0\rbrace
\end{align*}

To take a closer look at the different types of singularities, we need to consider four sub-cases.

(a) (Unique parabolic fixed point with $k=1$: non-singular point). The treatment of this case is similar to that of Case 1(a).

(b) (Unique parabolic fixed point with $k=3$: triple point). The only parameter for which $f_{a,b,c}$ has a parabolic cycle with three petals is $(0,1,0)$. A simple yet lengthy computation shows that $(0,1,0)$ is a singular point of $\mathrm{Per}_1(1)$; in particular, it is a triple point with three coincident tangents (or a tangent of multiplicity $3$). A heuristic reasoning for this kind of singularity can be given along the same lines of Case 1(b).

(c) (Non-isolated singularities). Since $\mathrm{Per}_1(1)$ is a two-dimensional algebraic set, it is perhaps not surprising that it has non-isolated singularities along a complex one-dimensional algebraic subset. These non-isolated singular parameters are given by the intersection of the algebraic sets 
\begin{center}
$V :=\lbrace (a,b,c)\in \mathbb{C}^3: \mathrm{sRes}_0(f_{a,b,c}(z)-z,f'_{a,b,c}(z)-1)=\mathrm{sRes}_1(f_{a,b,c}(z)-z,f'_{a,b,c}(z)-1)=0\rbrace$
\end{center}
This is, in fact, the intersection of $\mathrm{Per}_1(1)$ with $\lbrace (a,b,c)\in \mathbb{C}^3: 36 + 8 a^3 - 72 b + 36 b^2 - 32 a c=0\rbrace$. Each point of $V$ is a singular point of $\mathrm{Per}_1(1)$. 

As is clear from the defining property, for each $(a,b,c)\in V$, the corresponding polynomials $f_{a,b,c}(z)-z$ and $f'_{a,b,c}(z)-1$ have at least two common factors. The case where they have three common factors is already covered in the previous case, so we will now be concerned with the case $\mathrm{gcd}(f_{a,b,c}(z)-z,f'_{a,b,c}(z)-1)=2$.
This can happen in two different ways.

i) (Two parabolic fixed points each with $k=1$). Each parameter of the form $(a,1,\frac{a^2}{4})$ ($a\in \mathbb{C}^*$) belongs to $V$, and the corresponding polynomial $f_{a,1,\frac{a^2}{4}}$ has two distinct simple parabolic fixed points $\pm\sqrt{-\frac{a}{2}}$. This case is similar to Case 1(c); the singularities are formed by the transversal intersection of two branches of the algebraic set. Hence, each $(a,1,\frac{a^2}{4})$ ($a\in \mathbb{C}^*$) is a double point, and there are two tangents at each such singularity. 

ii) (Unique parabolic fixed point with $k=2$). For each $(a,b,c)\in V\setminus \lbrace (a,1,\frac{a^2}{4}): a\in \mathbb{C}\rbrace$, $f_{a,b,c}$ has a unique parabolic fixed point with two petals. Each such parameter is a double point with a single tangent of multiplicity $2$. A heuristic reasoning for this kind of singularity can be given along the same lines of Case 1(b).

Let us summarize our observations regarding the singular locus of $\mathrm{Per}_1(1)$ of degree $4$ monic centered polynomials. The singular locus $V$ admits a natural stratification $V = V_0\cup V_1\cup V_2$, where
\begin{eqnarray*}
V_0 &=& \lbrace (0,1,0)\rbrace ,\\
V_1 &=& \lbrace (a,b,c)\in \mathbb{C}^3: b=1,\ a^2-4c=0\rbrace \setminus \lbrace (0,1,0)\rbrace ,\\
V_2 &=& \lbrace (a,b,c)\in \mathbb{C}^3: 8a^3+27(1-b)^2=0,\ a^2+12c=0\rbrace \setminus \lbrace (0,1,0)\rbrace .
\end{eqnarray*}

Here, $(0,1,0)$ is a triple point of $\mathrm{Per}_1(1)$ with a tangent of multiplicity $3$, each point of $V_1$ is a double point with two distinct tangents, and each point of $V_2$ is a double point with a tangent of multiplicity $2$. Therefore, the Milnor fibrations at any two points of $V_i$ have the same fibration type. 

In general, we should ask the following questions, which are clearly motivated by the preceding analysis. 

\begin{question}
Let $\mathcal{F}=\lbrace f_{a_1, a_2, \cdots, a_m}\rbrace_{(a_1, a_2, \cdots, a_m) \in \mathbb{C}^m}$ be a holomorphic family of holomorphic polynomials of degree $d\geq 2$, depending algebraically on parameters $(a_1, a_2,$ $\cdots, a_m)\in \mathbb{C}^m$. 
\begin{enumerate}
\item Does every singularity of $\mathrm{Per}_n(1)$ occur either due to orbit mergers (giving rise to cusp-type singularities) or due to transverse intersection of more than one branches, each defining a periodic orbit of multiplier $1$ (giving rise to node-type singularities)? 

\item Can we classify the types of singularities of $\mathrm{Per}_n(1)$ in terms of dynamical properties of the singular parameters?

\item Study the topology of $\mathrm{Per}_n(1)$ near its singularities, especially near the non-isolated ones.
\end{enumerate}
\end{question}


\begin{thebibliography}{99}

\bibitem{B1}
 \newblock  A. F. Beardon,
     \newblock \emph{Iteration of Rational Functions},
     \newblock Complex Analytic Dynamical Systems Series: Graduate Texts in Mathematics, Vol. 132, Springer-Verlag, 1991.
     
\bibitem{BBM1}
\newblock A. Bonifant, X. Buff and J. Milnor,
\newblock Antipode preserving cubic maps: the fjord theorem, arXiv: 1512.01850.

\bibitem{BBM2}
\newblock A. Bonifant, X. Buff and J. Milnor,
\newblock Antipode preserving cubic maps II: tongues and the ring locus,
\newblock work in progress.

\bibitem[BPC]{BPC} 
\newblock S. Basu, R. Pollack and M.-F. Coste-Roy,
\newblock \emph{Algorithms in Real Algebraic Geometry},
\newblock 2nd edition, Springer, Berlin, 2006.

\bibitem{CFG}
\newblock J. Canela, N. Fagella and A. Garijo,
\newblock On a family of rational perturbations of the doubling map,
\newblock \emph{Journal of Difference Equations and Applications}, \textbf{21} (2015), 715--741.

\bibitem{BeEr}
\newblock W. Bergweiler and A. Eremenko,
\newblock Green's function and anti-holomorphic dynamics on a torus,
\newblock \emph{Proc. Amer. Math. Soc.}, \textbf{144} (2016), 2911--2922.

     
\bibitem{DU}   
\newblock M. Denker and M. Urbański, 
     \newblock Hausdorff and conformal measures on Julia sets with a rationally indifferent periodic point,
     \newblock \emph{J. London Math. Soc.}, \textbf{43(2)} (1991), 107--118.
     
\bibitem{DuFo}
 \newblock D. S. Dummit and R. M. Foote,
     \newblock \emph{Abstract Algebra},
     \newblock 3rd Edition, John Wiley and Sons, Inc., 2003.
     

\bibitem{HS}
    \newblock J. H. Hubbard and D. Schleicher,
    \newblock  Multicorns are not path connected,
    \newblock  in \emph{Frontiers in Complex Dynamics: In Celebration of John Milnor's 80th Birthday} (eds. A. Bonifant, M. Lyubich and S. Sutherland),
                Princeton University Press, (2014), 73--102.
                
\bibitem{IK}
   \newblock  H. Inou and J. Kiwi,
   \newblock Combinatorics and topology of straightening maps, I: compactness and bijectivity,
   \newblock \emph{Advances in Mathematics}, \textbf{231} (2012), 2666--2733.
                
\bibitem{IM}
\newblock H. Inou and S. Mukherjee,
\newblock Non-landing parameter rays of the multicorns,
\newblock \emph{Inventiones Mathematicae}, \textbf{204} (2016), 869--893.

\bibitem{IM1}
\newblock H. Inou and S. Mukherjee,
\newblock Discontinuity of straightening in antiholomorphic dynamics, arXiv: 1605.08061.

\bibitem{K1}
\newblock Frances Kirwan,
\newblock \emph{Complex Algebraic Curves},
\newblock Cambridge University Press, Cambridge, 1992.


\bibitem{JRL}   
\newblock J. Rivera-Letelier, 
     \newblock A connecting lemma for rational maps satisfying a no-growth condition,
     \newblock \emph{Ergodic Theory and Dynamical Systems}, \textbf{27(2)} (2007), 595--636.

\bibitem{Mc}   
\newblock C. T. McMullen, 
     \newblock Hausdorff dimension and conformal dynamics II: geometrically finite rational maps,
     \newblock \emph{Commentarii Mathematici Helvetici}, \textbf{75(4)} (2000), 535--593.
 
\bibitem{M1new}
 \newblock  J. Milnor,
     \newblock \emph{Dynamics in one complex variable},
     \newblock 3rd Edition, Princeton University Press, New Jersey, 2006.
     
\bibitem{M3}   
\newblock J. Milnor, 
     \newblock Remarks on iterated cubic maps,
     \newblock \emph{Experiment. Math.}, \textbf{1(1)} (1992), 5--24.
     
\bibitem{M5}
\newblock J. Milnor,
\newblock \emph{Singular Points of Complex Hypersurfaces},
\newblock Annals of Mathematics Studies. Princeton University Press, New Jersey, 1968.
     
\bibitem{MNS}
\newblock S. Mukherjee, S. Nakane and D. Schleicher,
\newblock On multicorns and unicorns II: bifurcations in spaces of antiholomorphic polynomials,
\newblock \emph{Ergodic Theory and Dynamical Systems}, \textbf{37} (2017), 859--899.

\bibitem{MU}
 \newblock  D. R. Mauldin and M. Urbański,
     \newblock \emph{Graph Directed Markov Systems: Geometry and Dynamics of Limit Sets},
     \newblock Cambridge University Press, Cambridge, 2003.

\bibitem{Na1}   
\newblock S. Nakane, 
     \newblock Connectedness of the tricorn,
     \newblock \emph{Ergodic Theory and Dynamical Systems}, \textbf{13} (1993), 349--356.
    
\bibitem{ND}   
\newblock N. Dobbs, 
     \newblock Nice sets and invariant densities in complex dynamics,
     \newblock \emph{Math. Proc. Cambridge Philos. Soc.}, \textbf{150(1)} (2011), 157--165.
    
\bibitem{NS}   
\newblock S. Nakane and D. Schleicher, 
     \newblock On multicorns and unicorns I: antiholomorphic dynamics, hyperbolic components, and real cubic polynomials,
     \newblock \emph{International Journal of Bifurcation and Chaos}, \textbf{13} (2003), 2825--2844.
    
\bibitem{Ru}   
\newblock D. Ruelle, 
     \newblock Repellers for real analytic maps,
     \newblock \emph{Ergodic Theory and Dynamical Systems}, \textbf{2(1)} (1982), 99--107.
    
    \bibitem{Sa}
\newblock S. Mukherjee,
\newblock Orbit portraits of unicritical antiholomorphic polynomials,
\newblock \emph{Conformal Geometry and Dynamics of the AMS}, \textbf{19} (2015), 35--50.

\bibitem{SU}   
\newblock B. Skorulski and M. Urbański, 
     \newblock Finer fractal geometry for analytic families of conformal dynamical systems,
     \newblock \emph{Dynamical Systems}, \textbf{29} (2014), 369--398.
    
    
\bibitem{U}   
\newblock M. Urbański, 
     \newblock Measures and dimensions in conformal dynamics,
     \newblock \emph{Bull. Amer. Math. Soc.}, \textbf{40} (2003), 281--321.
    
\bibitem{W}   
\newblock P. Walters,
     \newblock A variational principle for the pressure of continuous transformations,
     \newblock \emph{Amer. J. Math.}, \textbf{97} (1979), 937--971.
    
\bibitem{W1}
 \newblock  P. Walters,
     \newblock \emph{An introduction to ergodic theory},
     \newblock Graduate Texts in Mathematics, Volume 79, Springer, 1982.
     
\bibitem{W2} 
\newblock C. T. C. Wall,
\newblock \emph{Singular Points of Plane Curves},
\newblock London Mathematical Society Student Texts (vol. 63), Cambridge University Press, 2004.    

\bibitem{Zi}
 \newblock  M. Zinsmeister,
     \newblock \emph{Thermodynamic Formalism and Holomorphic Dynamical Systems},
     \newblock SMF/AMS Texts and Monographs, Volume 2, 2000.

\end{thebibliography}
\end{document}